\documentclass[11pt,a4paper]{article}

\usepackage[onehalfspacing]{setspace}

\usepackage{vmargin}         
\setmarginsrb{4cm}{2.5cm}{4cm}{1.5cm}{0cm}{0cm}{0cm}{1.5cm}

\usepackage[english]{babel}
\usepackage{enumitem}
\usepackage{url}

\usepackage{amsfonts}
\usepackage{amssymb}
\usepackage{amsmath}
\usepackage{amsthm}
\usepackage{extarrows}
\usepackage{tikz}
\usetikzlibrary{matrix,arrows}

\numberwithin{equation}{section}

\theoremstyle{plain}

\newtheorem{theorem}{Theorem}[section]
\newtheorem{proposition}[theorem]{Proposition}

\theoremstyle{definition} 

\newtheorem{definition}[theorem]{Definition}

\theoremstyle{remark} 

\newtheorem{remark}[theorem]{Remark}

\newcommand{\C}{\mathbb{C}}
\newcommand{\R}{\mathbb{R}}

\newcommand{\Proj}{\mathbb{P}}

\newcommand{\Op}{\ensuremath{\mathcal{O}}}

\DeclareMathOperator{\Spec}{Spec}
\DeclareMathOperator{\Spf}{Spf}

\begin{document}

\title{Local rigid cohomology of singular points}
\author{David Ouwehand}
\date{\today}
\maketitle
\begin{abstract}
\noindent
We show that if two singular points $x' \in X'$ and $x \in X$ on schemes over a field $k$ of characteristic $p > 0$ are contact equivalent then the rigid cohomology spaces $H^{\bullet}_{rig, \{x\}}(X)$ and $H^{\bullet}_{rig, \{x'\}}(X')$ are isomorphic. The isomorphism that we construct is moreover compatible with the Frobenius structure on rigid cohomology.
\end{abstract}

\section{Introduction and statement of results}

Throughout this paper we fix a perfect base field $k$ of characteristic $p > 0$. We also fix a Frobenius map $x \mapsto x^{p^r}$ on $k$, for some $r \geq 1$.

Let $X$ be a $k$-scheme together with a closed point $x \in X$. Now consider the space $H^i_{rig, \{x\}}(X)$, the rigid cohomology of $X$ with support in the closed set $\{x\}$. We will call this the \emph{local rigid cohomology} of $X$ at $x$. The aim of this paper is to show that local rigid cohomology behaves well with respect to contact equivalence. Indeed, as a direct consequence of our main theorem \ref{maintheorem} we will prove the following property.
\begin{theorem}
\label{theorem_constant_coeffs}
Let $x' \in X'$ and $x \in X$ be two closed points on $k$-schemes. Assume that these two points are contact equivalent. Then for all $i$ there exists an isomorphism
\[
H^i_{rig, \{x\}} \left( X \right) \xlongrightarrow{\sim} H^i_{rig, \{x'\}} \left( X' \right)
\]
on the local rigid cohomology. This isomorphism is moreover compatible with the Frobenius action on rigid cohomology.
\end{theorem}
This implies in particular that the local rigid cohomology is only interesting at singular points.

The rest of the introduction is organized as follows. First we introduce some notations and we impose a few technical assumptions on the objects that we study. In paragraph \ref{paragraph_equivalence} we recall the notion of contact equivalence for schemes over arbitrary fields. We also recall an important reformulation of this definition that is due to Artin. In paragraph \ref{paragraph_base_change} we recall how the functoriality of rigid cohomology and its canonical Frobenius action are defined in terms of base change maps. After this we formulate our main theorem \ref{maintheorem}, which is a more refined version of theorem \ref{theorem_constant_coeffs}.

The proof of theorem \ref{maintheorem} will then be covered in section \ref{section_proof_main_theorem}.

\subsection{Conventions and notations}

Throughout this paper $K$ denotes a field of characteristic zero, complete w.r.t.\ a discrete valuation, with valuation ring $\mathcal{V}$ and perfect residue field $k$ of characteristic $p > 0$. We also assume that the Frobenius map $x \mapsto x^{p^r}$ on $k$ admits an isometric lift $\sigma \colon K \rightarrow K$. We fix such a lift throughout the remainder of this paper.

Every scheme $X$ will be assumed to be defined over $k$. We will tacitly assume all $k$-schemes to be reduced, of finite type and separated. Morphisms between $k$-schemes are assumed to be defined over $k$. Every closed subset $C \subset X$ will be equipped with the reduced subscheme structure.

Every formal scheme $P$ will be assumed to be separated and topologically of finite type over $\mathcal{V}$. For such a $P$ we may consider the generic fiber $P_K$, which is a quasi-compact separated rigid analytic space over $K$. See \cite{nicaise} for details.

Recall that a \emph{frame} is a series of immersions $(X \subset Y \subset P)$ where $X$ and $Y$ are schemes over $k$ and $P$ is a formal scheme over $\mathcal{V}$. The immersion $X \subset Y$ is required to be open and $Y \subset P$ is assumed to be a closed immersion of $Y$ into the closed fiber of $P$. We let $S$ denote the frame $(\Spec k \subset \Spec k \subset \Spf \mathcal{V})$. In this way, any frame $(X \subset Y \subset P)$ is a frame over $S$. Every morphism of frames will be assumed to be an $S$-morphism. A scheme $X$ is called \emph{realizable} if there exists a frame $(X \subset Y \subset P)$ with $Y$ proper and $P$ smooth in a neighbourhood of $X$. All quasi-projective schemes are obviously realizable. It is easy to show that for any morphism of realizable schemes $f \colon X' \rightarrow X$ there exists a morphism of frames
\begin{center}
\begin{tikzpicture}[descr/.style={fill=white,inner sep=2.5pt}]
\matrix(a)[matrix of math nodes, row sep=3em, column sep=2.5em, text height=1.5ex, text depth=0.25ex
]
{
   X' & Y' & P' \\
   X  & Y  & P \\
};
\path[->,font=\scriptsize] (a-1-1)
edge node[above]{}  (a-1-2);
\path[->,font=\scriptsize] (a-1-2)
edge node[above]{}  (a-1-3);
\path[->,font=\scriptsize] (a-2-1)
edge node[above]{}  (a-2-2);
\path[->,font=\scriptsize] (a-2-2)
edge node[above]{}  (a-2-3);
\path[->,font=\scriptsize] (a-1-2)
edge node[right]{}  (a-2-2);
\path[->,font=\scriptsize] (a-1-3)
edge node[right]{$u$}  (a-2-3);
\path[->,font=\scriptsize] (a-1-1)
edge node[right]{$f$}  (a-2-1);
\end{tikzpicture}
\end{center}
such that $(X' \subset Y' \subset P')$ resp.\ $(X \subset Y \subset P)$ is a realization of $X'$ resp.\ of $X$ and such that $u$ is smooth in a neighbourhood of $X'$. Such a morphism of frames is called a \emph{realization} of $f$. From now on we only consider realizable schemes. For the rest we will use the standard notation and terminology from rigid cohomology. Our main reference for this is \cite{lestum}.

\subsection{Equivalence of singularities}
\label{paragraph_equivalence}

Consider two points on $k$-schemes $x' \in X'$ and $x \in X$. We say that these two points are \emph{contact equivalent} if there exists an isomorphism $\widehat{\Op}_{X, x} \xlongrightarrow{\sim} \widehat{\Op}_{X', x'}$ on the completed local rings. We denote this by $(X', x') \sim_c (X, x)$. This definition of contact equivalence is valid over any base field. It is a well-known fact that over $\C$ one recovers the classical analytic definition. See \cite[Corollary 1.6]{artin2}.

In this paper we will need to use the following reformulation of contact equivalence. It is due to Artin.

\begin{proposition}
\label{prop_equiv_points}
Two points $x' \in X'$ and $x \in X$ are contact equivalent if and only if there exists another scheme $X''$ together with a point $x'' \in X''$ and two morphisms $f' \colon X'' \rightarrow X'$ and $f \colon X'' \rightarrow X$ such that:
\begin{enumerate}[label=\roman*)]
\item
$f'(x'') = x'$ and $f'$ induces an isomorphism $k(x') \xlongrightarrow{\sim} k(x'')$ on the residue fields.
\item
$f(x'') = x$ and $f$ induces an isomorphism $k(x) \xlongrightarrow{\sim} k(x'')$ on the residue fields.
\item
$f'$ and $f$ are \'{e}tale at $x''$.
\end{enumerate}
\end{proposition}

\begin{proof}
For the ``if'' part of the proposition it suffices to observe that conditions i) through iii) imply that the Henselizations $\Op_{X, x}^h$ and $\Op_{X', x'}^h$ are isomorphic. Indeed, another way to formulate the conditions of the proposition is that $X''$ is a neighbourhood of $x$ in the \'etale site $X_{\acute{e}t}$ as well as a neighbourhood of $x'$ in the \'etale site $X'_{\acute{e}t}$. Hence the local rings of $x$ and $x'$ w.r.t.\ the \'etale site are isomorphic. The isomorphism $\Op_{X, x}^h \cong \Op_{X', x'}^h$ implies that the completions of the local rings are isomorphic as well. 

The ``only if'' part is proved in \cite[Corollary 2.6]{artin}.
\end{proof}

In the statement of our main theorem we will use a reformulation of the \'etale characterization of contact equivalence. For this we make the following definition.

\begin{definition}
\label{def_strongly_equiv_points}
Consider two points on $k$-schemes $x' \in X'$ and $x \in X$. We will write $(X', x') \succ (X, x)$ if there exists an open neighbourhood $U_{x'}$ of $x'$ and a morphism $f \colon U_{x'} \rightarrow X$ such that:
\begin{enumerate}[label=\roman*)]
\item
$f(x') = x$ and $f$ induces an isomorphism $k(x) \xlongrightarrow{\sim} k(x')$ on the residue fields.
\item
$f$ is \'{e}tale at $x'$.
\item
$f^{-1}(x) = \{x'\}$.
\end{enumerate}
\end{definition}

By proposition \ref{prop_equiv_points} it is obvious that two closed points $x' \in X'$ and $x \in X$ are contact equivalent if and only if there exists a pair $(X'', x'')$ such that $(X'', x'') \succ (X', x')$ and $(X'', x'') \succ (X, x)$.

\subsection{Base change maps, functoriality and Frobenius}
\label{paragraph_base_change}

Before we state our main theorem \ref{maintheorem} we will review the \emph{base change maps} that are needed to understand how rigid cohomology behaves w.r.t.\ morphisms of schemes. The main reason for doing so is to introduce some convenient notation for section \ref{section_proof_main_theorem}. We also use a slightly different presentation than in \cite{lestum}. The reader can safely skip this paragraph and refer back to it later if needed.

Consider a commutative diagram of rigid analytic spaces
\begin{center}
\begin{tikzpicture}[descr/.style={fill=white,inner sep=2.5pt}]
\matrix(a)[matrix of math nodes, row sep=3em, column sep=2.5em, text height=1.5ex, text depth=0.25ex
]
{
   V' & W' \\
   V  & W  \\
};
\path[->,font=\scriptsize] (a-1-1)
edge node[above]{$\alpha'$}  (a-1-2);
\path[->,font=\scriptsize] (a-2-1)
edge node[above]{$\alpha$}  (a-2-2);
\path[->,font=\scriptsize] (a-1-2)
edge node[right]{$\beta$}  (a-2-2);
\path[->,font=\scriptsize] (a-1-1)
edge node[right]{$\beta'$}  (a-2-1);
\end{tikzpicture}
\end{center}
where $\beta$ and $\beta'$ are flat. Let $E$ be an $\Op_V$-module. Then there is a canonical \emph{base change map}
\begin{equation}
\label{equation_basic_bcm}
\beta^* \left( \mathbb{R} \alpha_* \right) E \longrightarrow \left( \mathbb{R} \alpha'_* \right) (\beta')^* E .
\end{equation}
For the construction of this base change map we refer to \cite[Tag 02N6]{stacks-project} or to paragraph XVII.2 of \cite{sga}. Also note that by the flatness assumption on $\beta$ and $\beta'$ we do not need to consider the left derived functors $\mathbb{L} \beta^*$ and $\mathbb{L} (\beta')^*$.

The map (\ref{equation_basic_bcm}) is the starting point for the definition of the base change map of rigid cohomology. Let
\begin{equation}
\label{equation_bcm_morphism_frames}
\begin{tikzpicture}[descr/.style={fill=white,inner sep=2.5pt}, baseline=(current  bounding  box.center)]
\matrix(a)[matrix of math nodes, row sep=3em, column sep=2.5em, text height=1.5ex, text depth=0.25ex
]
{
 X' & Y' & P' \\
 X & Y & P \\
};
\path[->,font=\scriptsize] (a-1-1)
edge node[above]{}  (a-1-2);
\path[->,font=\scriptsize] (a-1-2)
edge node[above]{}  (a-1-3);
\path[->,font=\scriptsize] (a-2-1)
edge node[above]{}  (a-2-2);
\path[->,font=\scriptsize] (a-2-2)
edge node[above]{}  (a-2-3);
\path[->,font=\scriptsize] (a-1-1)
edge node[right]{$f$}  (a-2-1);
\path[->,font=\scriptsize] (a-1-2)
edge node[right]{}  (a-2-2);
\path[->,font=\scriptsize] (a-1-3)
edge node[right]{$u$}  (a-2-3);
\end{tikzpicture}
\end{equation}
be a flat morphism of frames. Also choose two closed subschemes $C' \subset X'$ and $C \subset X$ such that $f^{-1}(C) \subset C'$. Let $E$ be a $j_X^{\dagger} \Op_{]Y[_P}$-module with an integrable connection over $K$. 

Since (\ref{equation_bcm_morphism_frames}) is flat we know by \cite[Corollary 3.3.6]{lestum} that there exists a strict neighbourhood $V'$ of $]X'[_{P'}$ in $]Y'[_{P'}$ such that the morphism $u_K \colon ]Y'[_{P'} \rightarrow ]Y[_{P}$ is flat on $V'$. It follows from \cite[Proposition 6.2.2]{lestum} that restricting to $V'$ has no effect on cohomology, therefore we may work as if $u_K$ were flat. We can then apply the base change map (\ref{equation_basic_bcm}) coming from the diagram 
\begin{center}
\begin{tikzpicture}[descr/.style={fill=white,inner sep=2.5pt}]
\matrix(a)[matrix of math nodes, row sep=3em, column sep=2.5em, text height=1.5ex, text depth=0.25ex
]
{
   ]Y'[_{P'} & ]Y[_P \\
   ]Y[_{P} & ]Y[_P \\
};
\path[->,font=\scriptsize] (a-1-1)
edge node[above]{$u_K$}  (a-1-2);
\path[->,font=\scriptsize] (a-2-1)
edge node[above]{$\mbox{Id}$}  (a-2-2);
\path[->,font=\scriptsize] (a-1-2)
edge node[right]{$\mbox{Id}$}  (a-2-2);
\path[->,font=\scriptsize] (a-1-1)
edge node[right]{$u_K$}  (a-2-1);
\end{tikzpicture}
\end{center}
to the de Rham complex $\underline{\Gamma}^{\dagger}_{C} E \otimes_{\Op_{]Y[_P}} \Omega^{\bullet}_{]Y[_P / K}$. In this way we obtain a morphism
\[
(u \star)_1 \colon \underline{\Gamma}^{\dagger}_{C} E \otimes_{\Op_{]Y[_P}} \Omega^{\bullet}_{]Y[_P / K} \longrightarrow \left( \R u_{K *} \right) u_K^{*} \left( \underline{\Gamma}^{\dagger}_{C} E \otimes_{\Op_{]Y[_P}} \Omega^{\bullet}_{]Y[_P / K} \right).
\]
By applying the morphism of functors $\mbox{Id} \rightarrow j_{X'}^{\dagger}$ to the pullback of the de Rham complex we obtain another morphism
\begin{equation*}
\begin{split}
(u \star)_2 \colon \left( \R u_{K *} \right) u_K^{*} \mathopen{}\left( \underline{\Gamma}^{\dagger}_{C} E \otimes_{\Op_{]Y[_P}} \right.\mathclose{} & \mathopen{}\left. \Omega^{\bullet}_{]Y[_P / K} \right)\mathclose{} \\[1ex]
& \longrightarrow \left( \R u_{K *} \right) u^{\dagger} \left( \underline{\Gamma}^{\dagger}_{C} E \otimes_{\Op_{]Y[_P}} \Omega^{\bullet}_{]Y[_P / K} \right).
\end{split}
\end{equation*}
We also have a map
\begin{equation}
\label{equation_map_de_rham}
u_K^{*} \left( \underline{\Gamma}^{\dagger}_{C} E \otimes_{\Op_{]Y[_P}} \Omega^{\bullet}_{]Y[_P / K} \right) \longrightarrow u_K^{*} \underline{\Gamma}^{\dagger}_{C} E \otimes_{\Op_{]Y'[_{P'}}} \Omega^{\bullet}_{]Y'[_{P'} / K}
\end{equation}
from the pullback of the de Rham complex to the de Rham complex of the pullback. This map can easily be constructed by combining the canonical arrow $u_K^{*} \Omega^1_{]Y[_P} \rightarrow \Omega^1_{]Y'[_{P'}}$ with the definition of the de Rham complex. See for example to introduction of \cite{katz_oda}. From (\ref{equation_map_de_rham}) we find another map
\begin{equation*}
\begin{split}
(u \star)_3 \colon \left( \R u_{K *} \right) u^{\dagger} \mathopen{}\left( \underline{\Gamma}^{\dagger}_{C} E \otimes_{\Op_{]Y[_P}} \right.\mathclose{} & \mathopen{}\left. \Omega^{\bullet}_{]Y[_P / K} \right)\mathclose{} \\[1ex]
& \longrightarrow \left( \R u_{K *} \right) u^{\dagger} \underline{\Gamma}^{\dagger}_{C} E \otimes_{\Op_{]Y'[_{P'}}} \Omega^{\bullet}_{]Y'[_{P'} / K}.
\end{split}
\end{equation*}
Finally we have a canonical map $u^{\dagger} \underline{\Gamma}^{\dagger}_{C} E \longrightarrow \underline{\Gamma}^{\dagger}_{C'} u^{\dagger} E$, which we will study more in paragraph \ref{paragraph_canonical}. From this map we obtain a morphism
\[
(u \star)_4 \colon \left( \R u_{K *} \right) u^{\dagger} \underline{\Gamma}^{\dagger}_{C} E \otimes_{\Op_{]Y'[_{P'}}} \Omega^{\bullet}_{]Y'[_{P'} / K} \longrightarrow \left( \R u_{K *} \right) \underline{\Gamma}^{\dagger}_{C'} u^{\dagger} E \otimes_{\Op_{]Y'[_{P'}}} \Omega^{\bullet}_{]Y'[_{P'} / K}.
\]
\begin{definition}
\label{definition_base_change}
The canonical map
\begin{equation}
\label{equation_base_change}
u \star \colon \underline{\Gamma}^{\dagger}_{C} E \otimes_{\Op_{]Y[_P}} \Omega^{\bullet}_{]Y[_P / K} \longrightarrow \left( \R u_{K *} \right) \underline{\Gamma}^{\dagger}_{C'} u^{\dagger} E \otimes_{\Op_{]Y'[_{P'}}} \Omega^{\bullet}_{]Y'[_{P'} / K}
\end{equation}
that is given by the composition
\[
u \star = (u \star)_4 \circ (u \star)_3 \circ (u \star)_2 \circ (u \star)_1
\]
is called the \emph{base change map of rigid cohomology} (with respect to the morphism of frames (\ref{equation_bcm_morphism_frames}) and the module $E$).
\end{definition}
Note that our definition of the base change map is slightly different from the definition that can be found in \cite[Proposition 6.2.6]{lestum}. Indeed, that definition first applies (\ref{equation_map_de_rham}) and then $\mbox{Id} \rightarrow j_{X'}^{\dagger}$. It is easy to verify that these definitions amount to the same thing: just write out the compatibility condition for the natural transformation $\mbox{Id} \rightarrow j_{X'}^{\dagger}$.

The compatibility of rigid cohomology w.r.t.\ morphisms can easily be understood in terms of base change maps. Indeed, consider a morphism of $k$-schemes as in the commutative diagram below:
\begin{center}
\begin{tikzpicture}[descr/.style={fill=white,inner sep=2.5pt}]
\matrix(a)[matrix of math nodes, row sep=3em, column sep=1.5em, text height=1.5ex, text depth=0.25ex
]
{
   X' &   & X \\
      & \Spec k &   \\
};
\path[->,font=\scriptsize] (a-1-1)
edge node[above]{$f$}  (a-1-3);
\path[->,font=\scriptsize] (a-1-1)
edge node[left]{$p'$}  (a-2-2);
\path[->,font=\scriptsize] (a-1-3)
edge node[right]{$p$}  (a-2-2);
\end{tikzpicture}
\end{center}
Also assume that the morphism of frames (\ref{equation_bcm_morphism_frames}) is a realization of $f$. Now consider an overconvergent $F$-isocrystal $\mathcal{F}$ on $X$ and let $E$ denote the realization of $\mathcal{F}$ on $(X \subset Y \subset P)$. Then $u^{\dagger} E$ is the realization of $f^* \mathcal{F}$ on $(X' \subset Y' \subset P')$. The \emph{canonical map}
\begin{equation}
\label{equation_canonical_map_rig_coh}
H^i_{rig, C} \left( X, \mathcal{F} \right) \xlongrightarrow{} H^i_{rig, C'} \left( X', f^* \mathcal{F} \right)
\end{equation}
on rigid cohomology is obtained by applying the derived pushforward functor $\R p_{K *}$ to the base change map (\ref{equation_base_change}) and then taking the $i$-th cohomology. For $\mathcal{F} = \Op_{X / K}$ this map expresses the following fact: rigid cohomology with constant coefficients is a functor.

It is also a classical fact that the canonical map (\ref{equation_canonical_map_rig_coh}) is compatible with the Frobenius action on rigid cohomology. To see this, let $F_{X'}$ and $F_X$ denote the absolute Frobenius on $X'$ resp.\ on $X$. Also let $\Phi \colon F_X^* \mathcal{F} \xlongrightarrow{\sim} \mathcal{F}$ denote the Frobenius structure on $\mathcal{F}$, where the Frobenius pullback functor is defined as in \cite[Definition 8.3.1]{lestum}. Now consider the following diagram:
\begin{center}
\begin{tikzpicture}[descr/.style={fill=white,inner sep=2.5pt}]
\matrix(a)[matrix of math nodes, row sep=3em, column sep=2.5em, text height=1.5ex, text depth=0.25ex
]
{
   \mathbb{R} p_{rig, C} \mathcal{F} & \mathbb{R} p_{rig, C} F_X^* \mathcal{F} & \mathbb{R} p_{rig, C} \mathcal{F} \\
   \mathbb{R} p'_{rig, C'} f^* \mathcal{F}  & \mathbb{R} p'_{rig, C'} F_{X'}^* f^* \mathcal{F} & \mathbb{R} p'_{rig, C'} f^* \mathcal{F} \\
};
\path[->,font=\scriptsize] (a-1-1)
edge node[above]{}  (a-1-2);
\path[->,font=\scriptsize] (a-1-2)
edge node[above]{$\Phi$}  (a-1-3);
\path[->,font=\scriptsize] (a-2-1)
edge node[above]{}  (a-2-2);
\path[->,font=\scriptsize] (a-2-2)
edge node[above]{$f^* \Phi$}  (a-2-3);
\path[->,font=\scriptsize] (a-1-2)
edge node[right]{}  (a-2-2);
\path[->,font=\scriptsize] (a-1-3)
edge node[right]{}  (a-2-3);
\path[->,font=\scriptsize] (a-1-1)
edge node[right]{}  (a-2-1);
\end{tikzpicture}
\end{center}
The rows of this diagram describe the Frobenius actions on $\mathbb{R} p_{rig, C} \mathcal{F}$ and on $\mathbb{R} p'_{rig, C'} f^* \mathcal{F}$. The vertical arrows all come from the base change map of the pullback along $f$. We have to check that this diagram is commutative. All the arrows of the leftmost square are essentially instances of the base change map (\ref{equation_base_change}). This square commutes because of the identity $f \circ F_{X'} = F_X \circ f$. The square on the right commutes because the base change map is compatible with morphisms of overconvergent isocrystals.

\subsection{Statement of the main theorem}

With the material of paragraphs \ref{paragraph_equivalence} and \ref{paragraph_base_change} we are now ready to formulate our main theorem, which is about rigid cohomology with general coefficients.
\begin{theorem}
\label{maintheorem}
Let $(X', x')$ and $(X, x)$ be two $k$-schemes with marked closed points such that $(X', x') \succ (X, x)$ via $f \colon U_{x'} \rightarrow X \colon x' \mapsto x$. Let $\mathcal{F} \in \mbox{F-Isoc}^{\dagger} \left( X / S \right)$ be a finitely presented overconvergent F-isocrystal on $X$. Then for every $i \geq 0$ the canonical map on rigid cohomology 
\[
H^i_{rig, \{x\}} \left( X, \mathcal{F} \right) \xlongrightarrow{} H^i_{rig, \{x'\}} \left( X', f^* \mathcal{F} \right)
\]
is an isomorphism.
\end{theorem}

\begin{remark}
~\begin{enumerate}[label=\roman*)]
\item
In the statement of the theorem we have implicitly chosen an extension of $f^* \mathcal{F}$ to all of $X'$. The choice of the extension is not important, since by \cite[Proposition 8.2.8]{lestum} the local rigid cohomology only depends on an open neighbourhood of the support.
\item
In the statement of theorem \ref{maintheorem} it is important that $x'$ and $x$ are closed points, otherwise the cohomology with support does not make sense. This is not a problem if one only considers \emph{isolated singularities}, i.e.\ if one assumes that the singular loci $X_{sing}$ and $X'_{sing}$ are zero-dimensional at $x$ resp.\ at $x'$. Indeed, the singular locus of a scheme is closed under specialization; see for example \cite[Lemma 2.4.11.(b)]{liu}. So for quasi-compact schemes, every isolated singularity is a closed point.
\item
Recall that in the definition of the rigid cohomology with constant coefficients of a scheme $X$ one starts by choosing a realization $(X \subset Y \subset P)$. In order to show that $H^i_{rig}(X)$ is independent of the choice of the realization, one can prove that every diagram
\begin{center}
\begin{tikzpicture}[descr/.style={fill=white,inner sep=2.5pt}]
\matrix(a)[matrix of math nodes, row sep=3em, column sep=2.5em, text height=1.5ex, text depth=0.25ex
]
{
   X' = X & Y' & P' \\
   X  & Y  & P \\
};
\path[->,font=\scriptsize] (a-1-1)
edge node[above]{}  (a-1-2);
\path[->,font=\scriptsize] (a-1-2)
edge node[above]{}  (a-1-3);
\path[->,font=\scriptsize] (a-2-1)
edge node[above]{}  (a-2-2);
\path[->,font=\scriptsize] (a-2-2)
edge node[above]{}  (a-2-3);
\path[->,font=\scriptsize] (a-1-2)
edge node[right]{$g$}  (a-2-2);
\path[->,font=\scriptsize] (a-1-3)
edge node[right]{$u$}  (a-2-3);
\path[->,font=\scriptsize] (a-1-1)
edge node[right]{$f = \mbox{Id}_X$}  (a-2-1);
\end{tikzpicture}
\end{center}
with $g$ proper and $u$ smooth in a neighbourhood of $X$ induces an isomorphism on the cohomology. This is done in \cite[Proposition 6.5.3]{lestum}. More specifically, it is the base change map (\ref{equation_base_change}) that induces the isomorphism. Our approach for theorem \ref{maintheorem} is to prove a local version of this result in the case where $f$ is an \'{e}tale morphism rather than the identity map on $X$. See proposition \ref{theorem_base_change_main} for more details.
\end{enumerate}
\end{remark}

We now show that theorem \ref{theorem_constant_coeffs} is just a special case of our main theorem.

\begin{proof}[Proof of Theorem \ref{theorem_constant_coeffs}]
Let $\Op_{X' / K}$ resp.\ $\Op_{X / K}$ denote the constant F-isocrystal on $X'$ resp.\ on $X$. Now choose two morphisms $f' \colon X'' \rightarrow X'$ and $f \colon X'' \rightarrow X$ that satisfy the conditions from proposition \ref{prop_equiv_points}. Then use theorem \ref{maintheorem} and the fact that $(f')^* \Op_{X' / K} = f^* \Op_{X / K} = \Op_{X'' / K}$ to obtain an isomorphism
\begin{equation}
\label{equation_proof_thm_constant_coeffs}
H^i_{rig, \{x\}} \left( X \right) \xlongrightarrow{\sim} H^i_{rig, \{x'\}} \left( X' \right).
\end{equation}   
We have explained in paragraph \ref{paragraph_base_change} that the canonical map on rigid cohomology is Frobenius-equivariant. Therefore the isomorphism (\ref{equation_proof_thm_constant_coeffs}) is Frobenius-equivariant as well.
\end{proof}

We end this paragraph with a few more remarks.

\begin{remark}
Theorem \ref{maintheorem} can be seen as an analogue of a well-known result of \'etale cohomology. Indeed, the corresponding statement for the cohomology of an \'etale sheaf is a special case of the excision theorem \cite[Proposition III.1.27]{milne}. Combining this with proposition \ref{prop_equiv_points}, we see that theorem \ref{theorem_constant_coeffs} also has an analogous statement for $\ell$-adic cohomology.
\end{remark}

\begin{remark}
Fr\'ed\'eric D\'eglise has pointed out that theorem \ref{theorem_constant_coeffs} can be proved in a more general setting, using the theory of motives. See his remarks in \cite{deglise}.
\end{remark}

\section{Proof of the main theorem}
\label{section_proof_main_theorem}

This section contains the proof of our main theorem \ref{maintheorem}. First we recall the definition of the canonical map on sheaves with supports. We show that this map is an isomorphism under certain conditions. After this we reformulate our main theorem \ref{maintheorem} in terms of base change maps. We then finish the proof in the last two paragraphs.

\subsection{The canonical map on sheaves with supports}
\label{paragraph_canonical}

In this paragraph we give more details about the map $(u \star)_4$ from definition \ref{definition_base_change}. We show that this map is an isomorphism under a fairly general assumption.

Consider a morphism of frames
\begin{center}
\begin{tikzpicture}[descr/.style={fill=white,inner sep=2.5pt}]
\matrix(a)[matrix of math nodes, row sep=3em, column sep=2.5em, text height=1.5ex, text depth=0.25ex
]
{
 X' & Y' & P' \\
 X & Y & P \\
};
\path[->,font=\scriptsize] (a-1-1)
edge node[above]{}  (a-1-2);
\path[->,font=\scriptsize] (a-1-2)
edge node[above]{}  (a-1-3);
\path[->,font=\scriptsize] (a-2-1)
edge node[above]{}  (a-2-2);
\path[->,font=\scriptsize] (a-2-2)
edge node[above]{}  (a-2-3);
\path[->,font=\scriptsize] (a-1-1)
edge node[right]{}  (a-2-1);
\path[->,font=\scriptsize] (a-1-2)
edge node[right]{}  (a-2-2);
\path[->,font=\scriptsize] (a-1-3)
edge node[right]{$u$}  (a-2-3);
\end{tikzpicture}
\end{center}
Recall that if $E$ is an $\Op_{]Y[_P}$-module then there is a canonical map 
\[
u_K^* j_X^{\dagger} E \longrightarrow j_{X'}^{\dagger} u_K^* E
\]
where $u_K \colon ]Y'[_{P'} \rightarrow ]Y[_P$ denotes the morphism on tubes that one gets by restriction from the morphism $P'_K \rightarrow P_K$. This map is an isomorphism if the morphism of frames is Cartesian. See \cite[Corollary 5.3.9]{lestum} for more details. We briefly recall how this map can be used to define a canonical map on sheaves with supports. This construction can also be found in \cite[Corollary 5.3.10]{lestum}. Let $C' \subset X'$ and $C \subset X$ be closed subschemes such that $f^{-1}(C) \subset C'$. Define $U' = X' \setminus C'$ and $U = X \setminus C$. Then we obtain the following commutative diagram with exact rows:
\begin{equation}
\label{diagram_map_on_kernels}
\begin{tikzpicture}[descr/.style={fill=white,inner sep=2.5pt}, baseline=(current  bounding  box.center)]
\matrix(a)[matrix of math nodes, row sep=3em, column sep=2.5em, text height=1.5ex, text depth=0.25ex
]
{
   & u_K^* \underline{\Gamma}^{\dagger}_{C} j_X^{\dagger} E & u_K^* j_X^{\dagger} E       & u_K^* j_U^{\dagger} E    & 0 \\
0  & \underline{\Gamma}^{\dagger}_{C'} j_{X'}^{\dagger} u_K^* E & j_{X'}^{\dagger} u_K^* E & j_{U'}^{\dagger} u_K^* E & 0 \\
};

\path[->,font=\scriptsize] (a-1-2)
edge node[above]{}  (a-1-3);
\path[->,font=\scriptsize] (a-1-3)
edge node[above]{}  (a-1-4);
\path[->,font=\scriptsize] (a-1-4)
edge node[above]{}  (a-1-5);

\path[->,font=\scriptsize] (a-2-1)
edge node[above]{}  (a-2-2);
\path[->,font=\scriptsize] (a-2-2)
edge node[above]{}  (a-2-3);
\path[->,font=\scriptsize] (a-2-3)
edge node[above]{}  (a-2-4);
\path[->,font=\scriptsize] (a-2-4)
edge node[above]{}  (a-2-5);

\path[->,font=\scriptsize] (a-1-3)
edge node[above]{}  (a-2-3);
\path[->,font=\scriptsize] (a-1-4)
edge node[above]{}  (a-2-4);
\end{tikzpicture}
\end{equation}

By the universal property of the kernel we now obtain a canonical morphism
\begin{equation}
\label{equation_can_map_1}
u_K^* \underline{\Gamma}^{\dagger}_{C} j_X^{\dagger} E \longrightarrow \underline{\Gamma}^{\dagger}_{C'} j_{X'}^{\dagger} u_K^* E.
\end{equation}
In the case where $E$ is a $j_X^{\dagger} \Op_{]Y[_P}$-module, composing with $j_{X'}^{\dagger}$ also gives a canonical map
\begin{equation}
\label{equation_can_map_2}
u^{\dagger} \underline{\Gamma}^{\dagger}_{C} E \longrightarrow \underline{\Gamma}^{\dagger}_{C'} u^{\dagger} E.
\end{equation}
The first step towards proving the main theorem \ref{maintheorem} is to show that the canonical map on sheaves with supports is an isomorphism if the morphism of frames is flat and if the supports are Cartesian.

\begin{proposition}
\label{supports_lemma1}
Let
\begin{center}
\begin{tikzpicture}[descr/.style={fill=white,inner sep=2.5pt}]
\matrix(a)[matrix of math nodes, row sep=3em, column sep=2.5em, text height=1.5ex, text depth=0.25ex
]
{
 X' & Y' & P' \\
 X & Y & P \\
};
\path[->,font=\scriptsize] (a-1-1)
edge node[above]{}  (a-1-2);
\path[->,font=\scriptsize] (a-1-2)
edge node[above]{}  (a-1-3);
\path[->,font=\scriptsize] (a-2-1)
edge node[above]{}  (a-2-2);
\path[->,font=\scriptsize] (a-2-2)
edge node[above]{}  (a-2-3);
\path[->,font=\scriptsize] (a-1-1)
edge node[right]{}  (a-2-1);
\path[->,font=\scriptsize] (a-1-2)
edge node[right]{}  (a-2-2);
\path[->,font=\scriptsize] (a-1-3)
edge node[right]{$u$}  (a-2-3);
\end{tikzpicture}
\end{center}
be a flat morphism of frames. Let $E$ be a $j_X^{\dagger} \Op_{]Y[_P}$-module. Also choose two closed subschemes $C' \subset X'$ and $C \subset X$ such that $C' = C \times_X X'$. Then the canonical map (\ref{equation_can_map_2}) is an isomorphism.
\end{proposition}

\begin{proof}
We know by \cite[Corollary 3.3.6]{lestum} that $u_K$ is flat on some strict neighbourhood $V'$ of $]X'[_{P'}$ in $]Y'[_{P'}$. Therefore, if $j_{V'}$ denotes the inclusion, then the functor $(u_K \circ j_{V'})^*$ is exact. Applying this to the short exact sequence from \cite[Proposition 5.2.11]{lestum} yields
\begin{center}
\begin{tikzpicture}[descr/.style={fill=white,inner sep=2.5pt}]
\matrix(a)[matrix of math nodes, row sep=3em, column sep=2.5em, text height=1.5ex, text depth=0.25ex
]
{
0  & j_{V'}^{-1} u_K^* \underline{\Gamma}^{\dagger}_{C} E & j_{V'}^{-1} u_K^* E & j_{V'}^{-1} u_K^* j_U^{\dagger} E    & 0 \\
};

\path[->,font=\scriptsize] (a-1-1)
edge node[above]{}  (a-1-2);
\path[->,font=\scriptsize] (a-1-2)
edge node[above]{}  (a-1-3);
\path[->,font=\scriptsize] (a-1-3)
edge node[above]{}  (a-1-4);
\path[->,font=\scriptsize] (a-1-4)
edge node[above]{}  (a-1-5);
\end{tikzpicture}
\end{center}
By using the exactness of the functor $j_{X'}^{\dagger}$ together with \cite[Proposition 5.1.13]{lestum} we obtain the following short exact sequence, which is the same as applying $j_{X'}^{\dagger}$ to the top row of the diagram (\ref{diagram_map_on_kernels}):
\begin{center}
\begin{tikzpicture}[descr/.style={fill=white,inner sep=2.5pt}]
\matrix(a)[matrix of math nodes, row sep=3em, column sep=2.5em, text height=1.5ex, text depth=0.25ex
]
{
0  & j_{X'}^{\dagger} u_K^* \underline{\Gamma}^{\dagger}_{C} E & j_{X'}^{\dagger} u_K^* E & j_{X'}^{\dagger} u_K^* j_U^{\dagger} E    & 0 \\
};

\path[->,font=\scriptsize] (a-1-1)
edge node[above]{}  (a-1-2);
\path[->,font=\scriptsize] (a-1-2)
edge node[above]{}  (a-1-3);
\path[->,font=\scriptsize] (a-1-3)
edge node[above]{}  (a-1-4);
\path[->,font=\scriptsize] (a-1-4)
edge node[above]{}  (a-1-5);
\end{tikzpicture}
\end{center}
Therefore it is now sufficient to show that the map 
\begin{equation}
\label{supports_lemma1_map1}
j_{X'}^{\dagger} u_K^* j_U^{\dagger} E \xlongrightarrow{} j_{U'}^{\dagger} u_K^* E
\end{equation}
is an isomorphism. Here we have also used that the functor $j_{X'}^{\dagger}$ is exact, hence preserving kernels. The map (\ref{supports_lemma1_map1}) is obtained by applying $j_{X'}^{\dagger}$ to the map $u_K^* j_U^{\dagger} E \rightarrow j_{U'}^{\dagger} u_K^* E$ coming from the morphism of frames
\begin{center}
\begin{tikzpicture}[descr/.style={fill=white,inner sep=2.5pt}]
\matrix(a)[matrix of math nodes, row sep=3em, column sep=2.5em, text height=1.5ex, text depth=0.25ex
]
{
 U' & Y' & P' \\
 U & Y & P \\
};
\path[->,font=\scriptsize] (a-1-1)
edge node[above]{}  (a-1-2);
\path[->,font=\scriptsize] (a-1-2)
edge node[above]{}  (a-1-3);
\path[->,font=\scriptsize] (a-2-1)
edge node[above]{}  (a-2-2);
\path[->,font=\scriptsize] (a-2-2)
edge node[above]{}  (a-2-3);
\path[->,font=\scriptsize] (a-1-1)
edge node[right]{}  (a-2-1);
\path[->,font=\scriptsize] (a-1-2)
edge node[right]{}  (a-2-2);
\path[->,font=\scriptsize] (a-1-3)
edge node[right]{$u$}  (a-2-3);
\end{tikzpicture}
\end{center}
This morphism may be factored as 
\begin{center}
\begin{tikzpicture}[descr/.style={fill=white,inner sep=2.5pt}]
\matrix(a)[matrix of math nodes, row sep=3em, column sep=2.5em, text height=1.5ex, text depth=0.25ex
]
{
    U'	              \\
    U'' & Y' & P'\\
    U  & Y & P  \\
};
\path[->,font=\scriptsize] (a-2-1)
edge node[above]{}  (a-2-2);
\path[->,font=\scriptsize] (a-2-2)
edge node[above]{}  (a-2-3);
\path[->,font=\scriptsize] (a-3-1)
edge node[above]{}  (a-3-2);
\path[->,font=\scriptsize] (a-3-2)
edge node[above]{}  (a-3-3);
\path[->,font=\scriptsize] (a-2-2)
edge node[right]{}  (a-3-2);
\path[->,font=\scriptsize] (a-2-1)
edge node[left]{}  (a-3-1);
\path[->,font=\scriptsize] (a-1-1)
edge node[left]{}  (a-2-1);
\path[->,font=\scriptsize] (a-1-1)
edge node[right]{}  (a-2-2);
\path[->,font=\scriptsize] (a-2-3)
edge node[right]{$u$}  (a-3-3);
\end{tikzpicture}
\end{center}
where $U'' = U \times_Y Y'$ and where $U' \rightarrow U''$ is an open immersion. The map (\ref{supports_lemma1_map1}) is then the same thing as applying $j_{X'}^{\dagger}$ to the composition 
\[
u_K^* j_U^{\dagger} E \xlongrightarrow{} j_{U''}^{\dagger} u_K^* E \xlongrightarrow{} j_{U'}^{\dagger} u_K^* E .
\]
The first one of these arrows is an isomorphism by \cite[Corollary 5.3.9]{lestum}. Applying $j_{X'}^{\dagger}$ to the second arrow corresponds to the canonical map
\begin{equation}
\label{supports_lemma1_map2}
j^{\dagger}_{(X' \cap U'')} u_K^* E \xlongrightarrow{} j_{U'}^{\dagger} u_K^* E
\end{equation}
coming from the open immersion $U' \hookrightarrow X' \cap U''$. Note that both $X'$ and $U''$ are open subsets of $X'' = X \times_Y Y'$, which is itself an open subset of $Y'$. We may now write $U'' = X'' \setminus C''$, where $C'' = C \times_X X''$. But by assumption we have $C' = C \times_X X'$ and therefore $C' = C'' \times_{X''} X'$. This implies that 
\[
U'' \cap X' = U'' \times_{X''} X' = U' .
\]
It follows that the map (\ref{supports_lemma1_map2}) is an isomorphism, which finishes the proof.
\end{proof}

\subsection{Part I of the proof: Reformulation}
\label{paragraph_proof_0}

As we mentioned before, the key to proving our main theorem \ref{maintheorem} is to modify \cite[Proposition 6.5.3]{lestum}. The big difference is that in our setting we can only obtain a local result. More specifically, let $x' \in X'$ and $x \in X$ be closed points such that $(X', x') \succ (X, x)$ via some map $f \colon U_{x'} \rightarrow X$. Since the local cohomology at $x'$ only depends on an open neighbourhood of $x'$, we may assume that $U_{x'} = X'$ and that $f$ is \'{e}tale. What we need to show is that for such an $f$, the base change map with $C' = \{x'\}$ and $C = \{x\}$ is an isomorphism. See proposition \ref{theorem_base_change_main} below for the precise statement. If we assume that the morphism of frames in the statement of this proposition is a realization of $f$ and that $E$ is a realization of $\mathcal{F}$ then theorem \ref{maintheorem} follows immediately (c.f.\ the remarks at the end of paragraph \ref{paragraph_base_change}).

\begin{proposition}
\label{theorem_base_change_main}
Let 
\begin{center}
\begin{tikzpicture}[descr/.style={fill=white,inner sep=2.5pt}]
\matrix(a)[matrix of math nodes, row sep=3em, column sep=2.5em, text height=1.5ex, text depth=0.25ex
]
{
   X' & Y' & P' \\
   X  & Y  & P \\
};
\path[->,font=\scriptsize] (a-1-1)
edge node[above]{}  (a-1-2);
\path[->,font=\scriptsize] (a-1-2)
edge node[above]{}  (a-1-3);
\path[->,font=\scriptsize] (a-2-1)
edge node[above]{}  (a-2-2);
\path[->,font=\scriptsize] (a-2-2)
edge node[above]{}  (a-2-3);
\path[->,font=\scriptsize] (a-1-2)
edge node[right]{$g$}  (a-2-2);
\path[->,font=\scriptsize] (a-1-3)
edge node[right]{$u$}  (a-2-3);
\path[->,font=\scriptsize] (a-1-1)
edge node[right]{$f$}  (a-2-1);
\end{tikzpicture}
\end{center}
be a proper smooth morphism of smooth $S$-frames. Also assume that $f$ is \'{e}tale. Let $E$ be a coherent $j_X^{\dagger} \Op_{]Y[_P}$-module with an integrable connection over $K$. Choose two closed points $x' \in X'$ and $x \in X$ such that $f^{-1}(x) = \{x'\}$ and such that $f$ induces an isomorphism $k(x) \xlongrightarrow{\sim} k(x')$ on the residue fields. Then the base change map
\[
u \star \colon \underline{\Gamma}^{\dagger}_{\{x\}} E \otimes_{\Op_{]Y[_P}} \Omega^{\bullet}_{]Y[_P / K} \longrightarrow \left( \R u_{K *} \right) \underline{\Gamma}^{\dagger}_{\{x'\}} u^{\dagger} E \otimes_{\Op_{]Y'[_{P'}}} \Omega^{\bullet}_{]Y'[_{P'} / K}
\]
is an isomorphism.
\end{proposition}

The proof of proposition \ref{theorem_base_change_main} will be covered in the next two paragraphs.

\subsection{Part II of the proof: The quasi-compact \'{e}tale case}
\label{paragraph_proof_1}

The aim of this paragraph is to prove proposition \ref{theorem_base_change_main} in the case of an \'{e}tale morphism of frames such that the induced morphism on tubes $u_K \colon ]Y'[_{P'} \rightarrow ]Y[_P$ is quasi-compact. 

In the case of constant coefficients (i.e.\ $E = j_X^{\dagger} \Op_{]Y[_P}$) this fact is proved in the appendix of \cite{finitude}. Our proof for the general case uses similar techniques as in \cite[Proposition A.10]{finitude}, but introduces two technical improvements. Firstly there is the matter of switching the order of certain functors, which seems to be implicit in \cite{finitude}. We will briefly discuss the required properties below. Secondly, some extra care is needed to show that the constructed isomorphism is really the same as the canonical base change map (\ref{equation_base_change}). Indeed, the result in \cite[Proposition A.10]{finitude} is used to make statements about the \emph{dimension} of rigid cohomology. But for our applications the Frobenius-equivariance is very important.

\begin{proposition}
\label{theorem_base_change_etale}
Let 
\begin{center}
\begin{tikzpicture}[descr/.style={fill=white,inner sep=2.5pt}]
\matrix(a)[matrix of math nodes, row sep=3em, column sep=2.5em, text height=1.5ex, text depth=0.25ex
]
{
   X' & Y' & P' \\
   X  & Y  & P \\
};
\path[->,font=\scriptsize] (a-1-1)
edge node[above]{}  (a-1-2);
\path[->,font=\scriptsize] (a-1-2)
edge node[above]{}  (a-1-3);
\path[->,font=\scriptsize] (a-2-1)
edge node[above]{}  (a-2-2);
\path[->,font=\scriptsize] (a-2-2)
edge node[above]{}  (a-2-3);
\path[->,font=\scriptsize] (a-1-2)
edge node[right]{}  (a-2-2);
\path[->,font=\scriptsize] (a-1-3)
edge node[right]{$u$}  (a-2-3);
\path[->,font=\scriptsize] (a-1-1)
edge node[right]{$f$}  (a-2-1);
\end{tikzpicture}
\end{center}
be an \'{e}tale morphism of smooth $S$-frames such that the induced morphism on tubes $u_K \colon ]Y'[_{P'} \rightarrow ]Y[_P$ is quasi-compact. Choose two closed points $x' \in X'$ and $x \in X$ such that $f^{-1}(x) = \{x'\}$ and such that $f$ induces an isomorphism $k(x) \xlongrightarrow{\sim} k(x')$ on the residue fields. Let $E$ be a coherent $j_X^{\dagger} \Op_{]Y[_P}$-module with an integrable connection over $K$. Then the base change map
\[
u \star \colon \underline{\Gamma}^{\dagger}_{\{x\}} E \otimes_{\Op_{]Y[_P}} \Omega^{\bullet}_{]Y[_P / K} \longrightarrow \left( \R u_{K *} \right) \underline{\Gamma}^{\dagger}_{\{x'\}} u^{\dagger} E \otimes_{\Op_{]Y'[_{P'}}} \Omega^{\bullet}_{]Y'[_{P'} / K}
\]
is an isomorphism.
\end{proposition}

Before we can give the proof of this result, we will need to recall the definition of a certain modification of the functor $\underline{\Gamma}^{\dagger}_{C}$. This definition also appears in \cite{finitude}.

\begin{definition}
\label{definition_gamma_eta}
Consider a frame $(X \subset Y \subset P)$ and let $C \subset Y$ be a closed subset. Define $U = X \setminus C$ and $Z = Y \setminus U$. Then for any $\eta < 1$ we may consider the open immersion
\[
i_{\eta} \colon ]Y[_{P} \setminus ]Z[_{P \eta}  \xlongrightarrow{} ]Y[_P .
\]
As in \cite[A.9]{finitude}, we define
\[
\underline{\Gamma}_{\eta} E = \mbox{Ker}(E \xlongrightarrow{} i_{\eta *} i_{\eta}^{-1} E)
\]
for any $\Op_{]Y[_P}$-module $E$.
\end{definition}

We briefly recall some fundamental properties of the functor $\underline{\Gamma}_{\eta}$. Our formulations are slightly different from the results mentioned in \cite[A.9]{finitude}, because we specialize everything to overconvergent modules.

\begin{proposition}
\label{prop_functors_basic_1}
Use notations as in definition \ref{definition_gamma_eta}. Then there are canonical isomorphisms
\begin{equation}
\label{equation_limit_1}
\mathop {\varinjlim}_\eta \limits \, i_{\eta *} i_{\eta}^{-1} E \xlongrightarrow{\sim} j_{U}^{\dagger} E .
\end{equation}
and
\begin{equation}
\label{equation_limit_2}
\mathop {\varinjlim}_\eta \limits \, \underline{\Gamma}_{\eta} \, j_{X}^{\dagger} E \xlongrightarrow{\sim} \underline{\Gamma}_{C}^{\dagger} j_{X}^{\dagger} E.
\end{equation}
\end{proposition}

\begin{proof}
The isomorphism (\ref{equation_limit_1}) is proved in \cite[A.9]{finitude}. The isomorphism (\ref{equation_limit_2}) follows by combining (\ref{equation_limit_1}) with the short exact sequence from \cite[Proposition 5.2.11]{lestum}. Here one also uses that filtered colimits commute with finite limits.
\end{proof}

For the next property we focus on the case where the support $C$ is a closed point.

\begin{proposition}
\label{prop_functors_basic_2}
Use notations as in definition \ref{definition_gamma_eta}. Assume that $C = \{x\}$ is a closed point. Write $W = ]\{x\}[_{P}$ and let $i_W \colon W \hookrightarrow ]Y[_P$ denote the inclusion map. Then for any $\eta < 1$ and for any $\Op_{]Y[_P}$-module $E$, the base change map 
\begin{equation}
\label{equation_adj_unit}
\underline{\Gamma}_{\eta} \, j_X^{\dagger} E \longrightarrow \left( \mathbb{R} i_{W *} \right) i_W^{*} \underline{\Gamma}_{\eta} \, j_{X}^{\dagger} E
\end{equation}
that is associated to the diagram
\begin{center}
\begin{tikzpicture}[descr/.style={fill=white,inner sep=2.5pt}]
\matrix(a)[matrix of math nodes, row sep=3em, column sep=2.5em, text height=1.5ex, text depth=0.25ex
]
{
   W & ]Y[_P \\
   ]Y[_P  & ]Y[_P  \\
};
\path[->,font=\scriptsize] (a-1-1)
edge node[above]{$i_W$}  (a-1-2);
\path[->,font=\scriptsize] (a-2-1)
edge node[above]{$\mbox{Id}$}  (a-2-2);
\path[->,font=\scriptsize] (a-1-2)
edge node[right]{$\mbox{Id}$}  (a-2-2);
\path[->,font=\scriptsize] (a-1-1)
edge node[right]{$i_W$}  (a-2-1);
\end{tikzpicture}
\end{center}
is an isomorphism. 
\end{proposition}

\begin{proof}
Consider the immersion $\iota \colon ]Z[_P \rightarrow ]Y[_P$. It is proved in \cite[A.9]{finitude} that the canonical map
\[
\underline{\Gamma}_{\eta} \, j_X^{\dagger} E \longrightarrow \left( \mathbb{R} \iota_{*} \right) \iota^{*} \underline{\Gamma}_{\eta} \, j_{X}^{\dagger} E
\]
is an isomorphism. Therefore it suffices to show that there is an identification
\begin{equation}
\label{eq_lemma_identification}
\left( \mathbb{R} \iota_{*} \right) \iota^{*} \underline{\Gamma}_{\eta} \, j_{X}^{\dagger} E \cong \left( \mathbb{R} i_{W *} \right) i_W^{*} \underline{\Gamma}_{\eta} \, j_{X}^{\dagger} E .
\end{equation}
To see this, we use that $\{x\} \subset Y$ is a closed subset. This allows us to write, according to \cite[Proposition 2.2.15]{lestum}:
\[
]Z[_P \, = \, ]Y \setminus X[_{P} \, \, \cup \, \, ]\{x\}[_{P} \, = \, ]Y \setminus X[_{P} \, \, \cup \, \, W .
\]
But by construction, the restriction of $j_{X}^{\dagger} E$ to $]Y \setminus X[_{P}$ is zero. The identification (\ref{eq_lemma_identification}) follows.
\end{proof}

Next we investigate how the functor $\underline{\Gamma}_{\eta}$ behaves w.r.t.\ morphisms of frames.

\begin{definition}
\label{def_canmap_eta}
Consider a morphism of frames
\begin{center}
\begin{tikzpicture}[descr/.style={fill=white,inner sep=2.5pt}]
\matrix(a)[matrix of math nodes, row sep=3em, column sep=2.5em, text height=1.5ex, text depth=0.25ex
]
{
   X' & Y' & P' \\
   X  & Y  & P \\
};
\path[->,font=\scriptsize] (a-1-1)
edge node[above]{}  (a-1-2);
\path[->,font=\scriptsize] (a-1-2)
edge node[above]{}  (a-1-3);
\path[->,font=\scriptsize] (a-2-1)
edge node[above]{}  (a-2-2);
\path[->,font=\scriptsize] (a-2-2)
edge node[above]{}  (a-2-3);
\path[->,font=\scriptsize] (a-1-2)
edge node[right]{$g$}  (a-2-2);
\path[->,font=\scriptsize] (a-1-3)
edge node[right]{$u$}  (a-2-3);
\path[->,font=\scriptsize] (a-1-1)
edge node[right]{$f$}  (a-2-1);
\end{tikzpicture}
\end{center}
together with supports $C' \subset X'$ and $C \subset X$ such that $f^{-1}(C) \subset C'$. Define $U' = X' \setminus C'$ and $U = X \setminus C$, and note that $f$ restricts to $f \colon U' \rightarrow U$. Also consider closed complements $Z' = Y' \setminus U'$ and $Z = Y \setminus U$, which satisfy $g^{-1}(Z) \subset Z'$.

Then we have the following commutative diagram, for any $\eta < 1$:
\begin{center}
\begin{tikzpicture}[descr/.style={fill=white,inner sep=2.5pt}]
\matrix(a)[matrix of math nodes, row sep=3em, column sep=2.5em, text height=1.5ex, text depth=0.25ex
]
{
   ]Y'[_{P'} \setminus ]Z'[_{P' \eta} & ]Y'[_{P'} \\
   ]Y[_{P} \setminus ]Z[_{P \eta}  & ]Y[_P  \\
};
\path[->,font=\scriptsize] (a-1-1)
edge node[above]{$i'_{\eta}$}  (a-1-2);
\path[->,font=\scriptsize] (a-2-1)
edge node[above]{$i_{\eta}$}  (a-2-2);
\path[->,font=\scriptsize] (a-1-2)
edge node[right]{$u_K$}  (a-2-2);
\path[->,font=\scriptsize] (a-1-1)
edge node[right]{}  (a-2-1);
\end{tikzpicture}
\end{center}
For an $\Op_{]Y[_P}$-module $E$ we may then consider the (underived) base change map:
\[
u_K^{*} i_{\eta *} i_{\eta}^{-1} E \xlongrightarrow{} i'_{\eta *} u_K^{*} i_{\eta}^{-1} E = i'_{\eta *} (i'_{\eta})^{-1} u_K^{*} E .
\]
Then, using the left exactness of $j_{X'}^{\dagger}$, we obtain a commutative diagram whose top row is a complex and whose bottom row is exact:
\begin{equation}
\label{diagram_def_canmap_eta}
\begin{tikzpicture}[descr/.style={fill=white,inner sep=2.5pt}, baseline=(current  bounding  box.center)]
\matrix(a)[matrix of math nodes, row sep=3em, column sep=2.5em, text height=1.5ex, text depth=0.25ex
]
{
   & j_{X'}^{\dagger} u_K^{*} \underline{\Gamma}_{\eta} E & j_{X'}^{\dagger} u_K^{*} E       & j_{X'}^{\dagger} u_K^{*} i_{\eta *} i_{\eta}^{-1} E  \\
0  & j_{X'}^{\dagger} \underline{\Gamma}'_{\eta} u_K^{*} E & j_{X'}^{\dagger} u_K^{*} E & j_{X'}^{\dagger} i'_{\eta *} (i'_{\eta})^{-1} u_K^{*} E \\
};

\path[->,font=\scriptsize] (a-1-2)
edge node[above]{}  (a-1-3);
\path[->,font=\scriptsize] (a-1-3)
edge node[above]{}  (a-1-4);

\path[->,font=\scriptsize] (a-2-1)
edge node[above]{}  (a-2-2);
\path[->,font=\scriptsize] (a-2-2)
edge node[above]{}  (a-2-3);
\path[->,font=\scriptsize] (a-2-3)
edge node[above]{}  (a-2-4);

\path[->,font=\scriptsize] (a-1-3)
edge node[above]{}  (a-2-3);
\path[->,font=\scriptsize] (a-1-4)
edge node[above]{}  (a-2-4);
\end{tikzpicture}
\end{equation}
By the universal property of the kernel we now obtain a map
\begin{equation}
\label{eq_canmap_eta}
j_{X'}^{\dagger} u_K^{*} \underline{\Gamma}_{\eta} E \xlongrightarrow{} j_{X'}^{\dagger} \underline{\Gamma}'_{\eta} u_K^{*} E .
\end{equation}
Finally, by passing to the limit $\eta \rightarrow 1$, one obtains a map
\begin{equation}
\label{eq_canmap_eta_limit}
\varinjlim_{\eta} \, j_{X'}^{\dagger} u_K^{*} \underline{\Gamma}_{\eta} E \xlongrightarrow{} \varinjlim_{\eta} \, j_{X'}^{\dagger} \underline{\Gamma}'_{\eta} u_K^{*} E .
\end{equation}
\end{definition}

\begin{proposition}
\label{eta_lemma_1}
Use the notations from definition \ref{def_canmap_eta} and assume that the morphism of frames is flat. Then the top row of the diagram (\ref{diagram_def_canmap_eta}) is exact. Also, for any $\eta < 1$, the object $j_{X'}^{\dagger} u_K^{*} \underline{\Gamma}_{\eta} E$ is isomorphic to the kernel of the map
\[
j_{X'}^{\dagger} u_K^{*} E  \xlongrightarrow{} j_{X'}^{\dagger} u_K^{*} i_{\eta *} i_{\eta}^{-1} E .
\]
\end{proposition}

\begin{proof}
We know by \cite[Corollary 3.3.6]{lestum} that $u_K$ is flat on some strict neighbourhood $V'$ of $]X'[_{P'}$ in $]Y'[_{P'}$. Therefore, if $j_{V'}$ denotes the inclusion, then the functor $(u_K \circ j_{V'})^*$ is exact. This results in a short exact sequence
\begin{center}
\begin{tikzpicture}[descr/.style={fill=white,inner sep=2.5pt}]
\matrix(a)[matrix of math nodes, row sep=3em, column sep=2.5em, text height=1.5ex, text depth=0.25ex
]
{
0  & j_{V'}^{-1} u_K^{*} \underline{\Gamma}_{\eta} E & j_{V'}^{-1} u_K^{*} E       & j_{V'}^{-1} u_K^{*} i_{\eta *} i_{\eta}^{-1} E  \\
};

\path[->,font=\scriptsize] (a-1-1)
edge node[above]{}  (a-1-2);
\path[->,font=\scriptsize] (a-1-2)
edge node[above]{}  (a-1-3);
\path[->,font=\scriptsize] (a-1-3)
edge node[above]{}  (a-1-4);
\end{tikzpicture}
\end{center}
By applying the left exact functor $j_{X'}^{\dagger} j_{V' *}$ and using \cite[Proposition 5.1.13]{lestum}, we indeed find an exact sequence
\begin{center}
\begin{tikzpicture}[descr/.style={fill=white,inner sep=2.5pt}]
\matrix(a)[matrix of math nodes, row sep=3em, column sep=2.5em, text height=1.5ex, text depth=0.25ex
]
{
0  & j_{X'}^{\dagger} u_K^{*} \underline{\Gamma}_{\eta} E & j_{X'}^{\dagger} u_K^{*} E       & j_{X'}^{\dagger} u_K^{*} i_{\eta *} i_{\eta}^{-1} E  \\
};

\path[->,font=\scriptsize] (a-1-1)
edge node[above]{}  (a-1-2);
\path[->,font=\scriptsize] (a-1-2)
edge node[above]{}  (a-1-3);
\path[->,font=\scriptsize] (a-1-3)
edge node[above]{}  (a-1-4);
\end{tikzpicture}
\end{center}
This finishes the proof.
\end{proof}

The previous proposition is useful for the following lemma, which is one of the main technical ingredients in the proof of proposition \ref{prop_maintheorem_proof_prelim} below.

\begin{proposition}
\label{eta_lemma_2}
Use the notations from definition \ref{def_canmap_eta}. Assume that the morphism of frames is flat and that $E$ is a $j_{X}^{\dagger} \Op_{]Y[_P}$-module. Then the map (\ref{eq_canmap_eta_limit}) is isomorphic to the map (\ref{equation_can_map_2}).
\end{proposition}

\begin{proof}
By considering the constructions in the proofs of propositions \ref{eta_lemma_1} and \ref{supports_lemma1}, it is sufficient to show the following statement: The map
\[
\varinjlim_{\eta} \, j_{X'}^{\dagger} u_K^{*} i_{\eta *} i_{\eta}^{-1} E \xlongrightarrow{} \varinjlim_{\eta} \, j_{X'}^{\dagger} i'_{\eta *} (i'_{\eta})^{-1} u_K^{*} E
\]
coming from the rightmost vertical arrow of diagram (\ref{diagram_def_canmap_eta}) is the same thing as the map
\begin{equation}
\label{targetmap}
j_{X'}^{\dagger} u_K^* j_U^{\dagger} E \xlongrightarrow{} j_{U'}^{\dagger} u_K^* E
\end{equation}
coming from the rightmost vertical arrow of diagram (\ref{diagram_map_on_kernels}). 

On modules, the functor $j_{X'}^{\dagger}$ is left adjoint to the forgetful functor, according to \cite[Proposition 5.3.1]{lestum}. This means that $j_{X'}^{\dagger}$ preserves colimits, so that we are reduced to the map
\[
j_{X'}^{\dagger} \varinjlim_{\eta} \, u_K^{*} i_{\eta *} i_{\eta}^{-1} E \xlongrightarrow{} j_{X'}^{\dagger} \varinjlim_{\eta} \, i'_{\eta *} (i'_{\eta})^{-1} u_K^{*} E .
\]
The limit $\eta \rightarrow 1$ also slides through $u_K^{*}$, which is a left adjoint. By proposition \ref{prop_functors_basic_1} we indeed recover the map (\ref{targetmap}).
\end{proof}

In the remainder of this paragraph we will consider the map
\begin{equation}
\label{eq_finalmap}
j_{X'}^{\dagger} u_K^{*} \underline{\Gamma}_{\eta} E \xlongrightarrow{} j_{X'}^{\dagger} \underline{\Gamma}'_{\eta} u_K^{*} E \xlongrightarrow{} \underline{\Gamma}'_{\eta} \, j_{X'}^{\dagger} u_K^{*} E
\end{equation}
where the first arrow is the map (\ref{eq_canmap_eta}) and the second arrow is obtained by applying $j_{X'}^{\dagger} \underline{\Gamma}'_{\eta}$ to the map $u_K^{*} E \rightarrow j_{X'}^{\dagger} u_K^{*} E$. Here we also use that the module $\underline{\Gamma}'_{\eta} \, j_{X'}^{\dagger} u_K^{*} E$ is already overconvergent, so there is no need to write the $j_{X'}^{\dagger}$ in front of it. Using our earlier results we easily obtain the following property.
 
\begin{proposition}
\label{eta_lemma_3}
Again use the notations from definition \ref{def_canmap_eta}. Assume that the following conditions hold:
\begin{enumerate}[label=\roman*)]
\item
The given morphism of frames is flat.
\item
$E$ is a $j_{X}^{\dagger} \Op_{]Y[_P}$-module.
\item
The supports satisfy $f^{-1}(C) = C'$.
\end{enumerate}
Then the map (\ref{eq_finalmap}) becomes an isomorphism when passing to the limit $\eta \rightarrow 1$.
\end{proposition}

\begin{proof}
For the map (\ref{eq_canmap_eta_limit}) this follows immediately by combining propositions \ref{eta_lemma_2} and \ref{supports_lemma1}. Using the previous techniques, it is easy to show that 
\[
j_{X'}^{\dagger} \underline{\Gamma}'_{\eta} u_K^{*} E \xlongrightarrow{} \underline{\Gamma}'_{\eta} \, j_{X'}^{\dagger} u_K^{*} E
\]
becomes the identity on $\underline{\Gamma}^{\dagger}_{C'} j_{X'}^{\dagger} u_K^{*} E$ in the limit $\eta \rightarrow 1$.
\end{proof}

We now use the results about the functor $\underline{\Gamma}_{\eta}$ and the map (\ref{eq_finalmap}) to prove a weak version of proposition \ref{theorem_base_change_etale}.

\begin{proposition}
\label{prop_maintheorem_proof_prelim}
Consider an \'{e}tale morphism of frames
\begin{center}
\begin{tikzpicture}[descr/.style={fill=white,inner sep=2.5pt}]
\matrix(a)[matrix of math nodes, row sep=3em, column sep=2.5em, text height=1.5ex, text depth=0.25ex
]
{
   X' & Y' & P' \\
   X  & Y  & P \\
};
\path[->,font=\scriptsize] (a-1-1)
edge node[above]{}  (a-1-2);
\path[->,font=\scriptsize] (a-1-2)
edge node[above]{}  (a-1-3);
\path[->,font=\scriptsize] (a-2-1)
edge node[above]{}  (a-2-2);
\path[->,font=\scriptsize] (a-2-2)
edge node[above]{}  (a-2-3);
\path[->,font=\scriptsize] (a-1-2)
edge node[right]{}  (a-2-2);
\path[->,font=\scriptsize] (a-1-3)
edge node[right]{$u$}  (a-2-3);
\path[->,font=\scriptsize] (a-1-1)
edge node[right]{$f$}  (a-2-1);
\end{tikzpicture}
\end{center}
such that the induced morphism on tubes $u_K \colon ]Y'[_{P'} \rightarrow ]Y[_P$ is quasi-compact. Choose two closed points $x' \in X'$ and $x \in X$ such that $f^{-1}(x) = \{x'\}$ and such that $f$ induces an isomorphism $k(x) \xlongrightarrow{\sim} k(x')$ on the residue fields. Let $E$ be a coherent $j_X^{\dagger} \Op_{]Y[_P}$-module. Then the canonical map 
\begin{equation}
\label{equation_theorem_base_change_etale_1}
\underline{\Gamma}^{\dagger}_{\{x\}} E \longrightarrow \left( \mathbb{R} u_{K *} \right) u_K^* \underline{\Gamma}^{\dagger}_{\{x\}} E \longrightarrow \left( \mathbb{R} u_{K *} \right) u^{\dagger} \underline{\Gamma}^{\dagger}_{\{x\}} E
\end{equation}
that is defined in a similar way as the composition $(u \star)_2 \circ (u \star)_1$ from paragraph \ref{paragraph_base_change} is an isomorphism. 
\end{proposition}

\begin{proof}
Define $W' = ]\{x'\}[_{P'}$ and $W = ]\{x\}[_P$. Then we have a commutative diagram
\begin{center}
\begin{tikzpicture}[descr/.style={fill=white,inner sep=2.5pt}]
\matrix(a)[matrix of math nodes, row sep=3em, column sep=2.5em, text height=1.5ex, text depth=0.25ex
]
{
   W' & ]Y'[_{P'} \\
   W  & ]Y[_P  \\
};
\path[->,font=\scriptsize] (a-1-1)
edge node[above]{$i_{W'}$}  (a-1-2);
\path[->,font=\scriptsize] (a-2-1)
edge node[above]{$i_{W}$}  (a-2-2);
\path[->,font=\scriptsize] (a-1-2)
edge node[right]{$u_K$}  (a-2-2);
\path[->,font=\scriptsize] (a-1-1)
edge node[right]{$v$}  (a-2-1);
\end{tikzpicture}
\end{center}
where $v$ is the restriction of $u_K$ and $i_{W'}$, $i_W$ are open immersions. The fact that $f$ induces an isomorphism $k(x) \xlongrightarrow{\sim} k(x')$ on residue fields means that the restriction $f \colon \Spec k(x') \rightarrow \Spec k(x)$ is an isomorphism. Since $u_K$ is \'{e}tale in a neighbourhood of $x'$ it then follows from \cite[Proposition 2.3.15]{lestum} that $v$ is an isomorphism.

We now apply a standard property about the behaviour of the base change map (\ref{equation_basic_bcm}) with respect to a composition of diagrams. This property is formulated for schemes in \cite[Proposition XII.4.4]{sga}. The proof is formally the same for any ringed space. We will apply this composition property to the diagram 
\begin{equation}
\label{equation_diagram_1}
\begin{tikzpicture}[descr/.style={fill=white,inner sep=2.5pt}, baseline=(current  bounding  box.center)]
\matrix(a)[matrix of math nodes, row sep=3em, column sep=2.5em, text height=1.5ex, text depth=0.25ex
]
{
   W' & ]Y'[_{P'} & ]Y[_P \\
   W  & ]Y[_P  & ]Y[_P \\
};
\path[->,font=\scriptsize] (a-1-1)
edge node[above]{$i_{W'}$}  (a-1-2);
\path[->,font=\scriptsize] (a-1-2)
edge node[above]{$u_K$}  (a-1-3);
\path[->,font=\scriptsize] (a-2-1)
edge node[above]{$i_W$}  (a-2-2);
\path[->,font=\scriptsize] (a-2-2)
edge node[above]{$\mbox{Id}$}  (a-2-3);
\path[->,font=\scriptsize] (a-1-2)
edge node[right]{$u_K$}  (a-2-2);
\path[->,font=\scriptsize] (a-1-3)
edge node[right]{$\mbox{Id}$}  (a-2-3);
\path[->,font=\scriptsize] (a-1-1)
edge node[right]{$v$}  (a-2-1);
\end{tikzpicture}
\end{equation}
and to the sheaf $i_W^* \underline{\Gamma}_{\eta} E$. By making use of the isomorphism from proposition \ref{prop_functors_basic_2} we obtain a commutative diagram
\begin{center}
\begin{tikzpicture}[descr/.style={fill=white,inner sep=2.5pt}]
\matrix(a)[matrix of math nodes, row sep=3em, column sep=2.5em, text height=1.5ex, text depth=0.25ex
]
{
   \underline{\Gamma}_{\eta} E & \left( \mathbb{R} u_{K *} \right) u_K^* \underline{\Gamma}_{\eta} E \\
     & \left( \mathbb{R} u_{K *} i_{W' *} \right) v^* i_W^* \underline{\Gamma}_{\eta} E \\
};
\path[->,font=\scriptsize] (a-1-1)
edge node[above]{$a_1$}  (a-1-2);
\path[->,font=\scriptsize] (a-1-2)
edge node[right]{$a_2$}  (a-2-2);
\path[->,font=\scriptsize] (a-1-1)
edge node[below]{$a_3$}  (a-2-2);
\end{tikzpicture}
\end{center}
The horizontal arrow $a_1$ is the base change map coming from the rightmost square of (\ref{equation_diagram_1}), applied to the sheaf $\left( \mathbb{R} i_{W *} \right) i_W^* \underline{\Gamma}_{\eta} E$. The map $a_2$ is obtained by applying $\mathbb{R} u_{K *}$ to the base change map coming from the leftmost square of (\ref{equation_diagram_1}). By using \cite[Proposition XII.4.4]{sga} again we see that $a_2$ is equal to the morphism that one gets after applying $\mathbb{R} u_{K *}$ to the canonical morphism 
\[
u_K^* \underline{\Gamma}_{\eta} E \longrightarrow \left( \mathbb{R} i_{W' *} \right) i_{W'}^* u_K^* \underline{\Gamma}_{\eta} E.
\]
The arrow $a_3$ is obtained by taking the base change map of the total diagram (\ref{equation_diagram_1}). Note that $a_3$ is equal to the morphism that one gets by applying $\mathbb{R} i_{W *}$ to the canonical map
\[
i_W^* \underline{\Gamma}_{\eta} E \longrightarrow \left( \mathbb{R} v_* \right) v^* i_W^* \underline{\Gamma}_{\eta} E.
\]
Since $v$ is an isomorphism it follows that $a_3$ is an isomorphism. Let us now look at the diagram
\begin{center}
\begin{tikzpicture}[descr/.style={fill=white,inner sep=2.5pt}]
\matrix(a)[matrix of math nodes, row sep=3em, column sep=2.5em, text height=1.5ex, text depth=0.25ex
]
{
   \left( \mathbb{R} u_{K *} \right) u_K^* \underline{\Gamma}_{\eta} E & \left( \mathbb{R} u_{K *} \right) j_{X'}^{\dagger} u_K^* \underline{\Gamma}_{\eta} E \\
   \left( \mathbb{R} u_{K *} i_{W' *} \right) i_{W'}^* u_K^* \underline{\Gamma}_{\eta} E & \left( \mathbb{R} u_{K *} i_{W' *} \right) i_{W'}^* j_{X'}^{\dagger} u_K^* \underline{\Gamma}_{\eta} E \\
};
\path[->,font=\scriptsize] (a-1-1)
edge node[above]{$b_1$}  (a-1-2);
\path[->,font=\scriptsize] (a-1-2)
edge node[right]{$b_2$}  (a-2-2);
\path[->,font=\scriptsize] (a-2-1)
edge node[below]{$b_3$}  (a-2-2);
\path[->,font=\scriptsize] (a-1-1)
edge node[right]{$a_2$}  (a-2-1);
\end{tikzpicture}
\end{center}
that is obtained by applying the canonical map
\[
\mbox{Id} \xlongrightarrow{} \left( \mathbb{R} i_{W' *} \right) i_{W'}^{*}
\]
to the morphism 
\[
u_K^* \underline{\Gamma}_{\eta} E \xlongrightarrow{} j_{X'}^{\dagger} u_K^* \underline{\Gamma}_{\eta} E
\]
and then composing with $\mathbb{R} u_{K *}$. 

Note that the map $b_3$ is an isomorphism. This follows from the characterization \cite[Proposition 5.1.12]{lestum} of the functor $j^{\dagger}_{X'}$. Indeed, $i_{W'}^*$ is a left adjoint, hence preserving filtered colimits. If $V'$ is any strict neighbourhood of $]X'[_{P'}$ in $]Y'[_{P'}$, then $W' \subset V'$, so that $i_{W'}^* j_{V' *} j_{V'}^{-1} = i_{W'}^*$. This shows that $i_{W'}^* j^{\dagger}_{X'} = i_{W'}^*$.

In a similar way we construct another diagram 
\begin{center}
\begin{tikzpicture}[descr/.style={fill=white,inner sep=2.5pt}]
\matrix(a)[matrix of math nodes, row sep=3em, column sep=2.5em, text height=1.5ex, text depth=0.25ex
]
{
   \left( \mathbb{R} u_{K *} \right) j_{X'}^{\dagger} u_K^* \underline{\Gamma}_{\eta} E  & \left( \mathbb{R} u_{K *} \right) \underline{\Gamma}'_{\eta} \, j_{X'}^{\dagger} u_K^{*} E \\
   \left( \mathbb{R} u_{K *} i_{W' *} \right) i_{W'}^* j_{X'}^{\dagger} u_K^* \underline{\Gamma}_{\eta} E & \left( \mathbb{R} u_{K *} i_{W' *} \right) i_{W'}^* \underline{\Gamma}'_{\eta} \, j_{X'}^{\dagger} u_K^{*} E \\
};
\path[->,font=\scriptsize] (a-1-1)
edge node[above]{$c_1$}  (a-1-2);
\path[->,font=\scriptsize] (a-1-2)
edge node[right]{$c_2$}  (a-2-2);
\path[->,font=\scriptsize] (a-2-1)
edge node[below]{$c_3$}  (a-2-2);
\path[->,font=\scriptsize] (a-1-1)
edge node[right]{$b_2$}  (a-2-1);
\end{tikzpicture}
\end{center}
using the morphism (\ref{eq_finalmap}). Observe that the map $c_2$ is an isomorphism, according to proposition \ref{prop_functors_basic_2}.

We now have the identity
\begin{equation}
\label{eqdecomp}
c_2 \circ c_1 \circ b_1 \circ a_1 = c_3 \circ b_3 \circ a_3
\end{equation}
and we have shown that $a_3$, $b_3$ and $c_2$ are isomorphisms. 

At this point, we fix an $m \geq 0$ and we consider the $m$-th cohomology of all the complexes above. It is sufficient to consider these modules, since the statement of the proposition is about a complex concentrated in degree zero. By abuse of notation, we still write the maps as $a_1, b_1, \ldots$. 

Note that (the $m$-th cohomology of) the map (\ref{equation_theorem_base_change_etale_1}) is recovered by taking the limit $\eta \rightarrow 1$ of the composition $b_1 \circ a_1$. To see this, first observe that $\varinjlim_{\eta}$ commutes with $\R^{m} u_{K *}$, since $u_K$ is assumed quasi-compact. See \cite[0.1.8]{berthelot} for details.  Also, $\varinjlim_{\eta}$ commutes with the functors $u_K^*$ and $j^{\dagger}_{X'}$, since these are left adjoints.

We have proved in proposition \ref{eta_lemma_3} that $c_1$ becomes an isomorphism in the limit $\eta \rightarrow 1$. We now show that $c_3$ has the same property. First note that $i_{W'}$ is quasi-Stein, since $W'$ is the tube of a closed subset of $Y'$. It follows that $\mathbb{R}^n i_{W' *} \mathcal{M} = 0$ for all $n > 0$ and for any coherent $\Op_{W'}$-module $\mathcal{M}$. Using the assumption that $E$ is coherent, one sees that $c_3$ is the same as the morphism
\[
\left( \mathbb{R}^m u_{K *}  \right) i_{W' *} i_{W'}^* j_{X'}^{\dagger} u_K^* \underline{\Gamma}_{\eta} E \xlongrightarrow{} \left( \mathbb{R}^m u_{K *} \right) i_{W' *}  i_{W'}^* \underline{\Gamma}'_{\eta} \, j_{X'}^{\dagger} u_K^{*} E .
\]
It now suffices to show that the limit $\eta \rightarrow 1$ commutes with $i_{W' *}$. But this follows immediately from the definitions, since $i_{W'}$ is the inclusion of an admissible open subset.

Once we know that $c_3$ is an isomorphism in the limit, it follows immediately that also $b_1 \circ a_1$ becomes an isomorphism for $m=0$ and $\eta \rightarrow 1$.

Similarly, for $m > 0$ we know that 
\[
\left( \mathbb{R}^m u_{K *} \right) i_{W' *} i_{W'}^* u_K^* \underline{\Gamma}_{\eta} E = 0. 
\]
But in the limit $\eta \rightarrow 1$ this module is isomorphic to
\[
\varinjlim_{\eta} \, \left( \mathbb{R}^m u_{K *} \right) j_{X'}^{\dagger} u_K^* \underline{\Gamma}_{\eta} E = \left( \mathbb{R}^m u_{K *} \right) j_{X'}^{\dagger} u_K^* \underline{\Gamma}^{\dagger}_{\{x\}} E ,
\]
which must also be zero. This completes the proof.
\end{proof}

With proposition \ref{prop_maintheorem_proof_prelim} in place the proof of proposition \ref{theorem_base_change_etale} becomes relatively straightforward.

\begin{proof}[Proof of Proposition \ref{theorem_base_change_etale}]
As we did in paragraph \ref{paragraph_base_change}, we divide the base change map $u \star$ into several parts $(u \star)_i$ for $i = 1, \ldots, 4$. It follows from proposition \ref{supports_lemma1} that $(u \star)_4$ is an isomorphism. By \cite[Corollary 3.3.6]{lestum} we know that the map $u_K \colon ]Y'[_{P'} \rightarrow ]Y[_P$ is \'{e}tale on some strict neighbourhood $V'$ of $]X'[_{P'}$ in $]Y'[_{P'}$. It then follows from \cite[Proposition 5.3.7]{lestum} that $(u \star)_3$ is an isomorphism if and only if the corresponding map for the restriction $u_K \colon V' \rightarrow ]Y[_P$ is an isomorphism. Therefore we may work as if $u_K$ were \'{e}tale. Since our frames are smooth, we may also work as if $]Y'[_{P'}$ and $]Y[_P$ were smooth. In this situation the canonical map $u_K^* \Omega^1_{]Y[_P} \rightarrow \Omega^1_{]Y'[_{P'}}$ is an isomorphism. It immediately follows that $(u \star)_3$ is an isomorphism as well. 

It remains to show that the composition $(u \star)_2 \circ (u \star)_1$ is an isomorphism. For this we can use proposition \ref{prop_maintheorem_proof_prelim}. Indeed, proposition \ref{prop_maintheorem_proof_prelim} is equivalent to saying that for every coherent $j_X^{\dagger} \Op_{]Y[_P}$-module $\mathcal{E}$, the map 
\begin{equation}
\label{equation_basechange_1}
\underline{\Gamma}^{\dagger}_{\{x\}} \mathcal{E} \longrightarrow u_{K *} u^{\dagger} \underline{\Gamma}^{\dagger}_{\{x\}} \mathcal{E} 
\end{equation}
is an isomorphism and
\begin{equation}
\label{equation_basechange_2}
\left( \mathbb{R}^m u_{K *} \right) u^{\dagger} \underline{\Gamma}^{\dagger}_{\{x\}} \mathcal{E} = 0
\end{equation}
for all $m > 0$. Note that by Proposition 5.3.2 and Corollary 5.3.3 in \cite{lestum}, every term of the de Rham complex $\underline{\Gamma}^{\dagger}_{\{x\}} E \otimes_{\Op_{]Y[_P}} \Omega^{\bullet}_{]Y[_P / K}$ satisfies the conditions of proposition \ref{prop_maintheorem_proof_prelim}. By using the spectral sequence of hypercohomology
\begin{equation*}
\begin{split}
E_1^{m, n} = \left( \mathbb{R}^m u_{K *} \right) u^{\dagger} \mathopen{}\left( \underline{\Gamma}^{\dagger}_{\{x\}} E \right.\mathclose{} & \mathopen{}\left. \otimes_{\Op_{]Y[_P}} \Omega^{n}_{]Y[_P / K} \right)\mathclose{} \Longrightarrow  \\
& \left( \mathbb{R}^{m + n} u_{K *} \right) u^{\dagger} \left( \underline{\Gamma}^{\dagger}_{\{x\}} E \otimes_{\Op_{]Y[_P}} \Omega^{\bullet}_{]Y[_P / K} \right)
\end{split}
\end{equation*}
we deduce from (\ref{equation_basechange_2}) that 
\[
\left( \mathbb{R} u_{K *} \right) u^{\dagger} \left( \underline{\Gamma}^{\dagger}_{\{x\}} E \otimes_{\Op_{]Y[_P}} \Omega^{\bullet}_{]Y[_P / K} \right) = u_{K *} u^{\dagger} \left( \underline{\Gamma}^{\dagger}_{\{x\}} E \otimes_{\Op_{]Y[_P}} \Omega^{\bullet}_{]Y[_P / K} \right).
\]
This means that $(u \star)_2 \circ (u \star)_1$ is equal to the canonical map
\begin{equation}
\label{equation_basechange_3}
\underline{\Gamma}^{\dagger}_{\{x\}} E \otimes_{\Op_{]Y[_P}} \Omega^{\bullet}_{]Y[_P / K} \longrightarrow u_{K *} u^{\dagger} \left( \underline{\Gamma}^{\dagger}_{\{x\}} E \otimes_{\Op_{]Y[_P}} \Omega^{\bullet}_{]Y[_P / K} \right).
\end{equation}
But this map can be computed  term by term. From the fact that (\ref{equation_basechange_1}) is an isomorphism it then follows that  (\ref{equation_basechange_3}) is an isomorphism as well. We have now shown that the maps $(u \star)_4$, $(u \star)_3$ and $(u \star)_2 \circ (u \star)_1$ are isomorphisms. This finishes the proof.
\end{proof}

\begin{remark}
We have shown that the composition $(u \star)_2 \circ (u \star)_1$ is an isomorphism. However, the individual maps $(u \star)_2$ and $(u \star)_1$ need not be isomorphisms. For this reason it was convenient to make definition \ref{definition_base_change} slightly different from the definition in \cite{lestum}.
\end{remark}

\subsection{Part III of the proof: The general case}
\label{paragraph_proof_2}

In this paragraph we finish the proof of proposition \ref{theorem_base_change_main}. First we improve proposition \ref{theorem_base_change_etale} by removing the condition that the induced map on tubes $u_K \colon ]Y'[_{P'} \rightarrow ]Y[_P$ should be quasi-compact. 

\begin{proposition}
\label{theorem_base_change_etale_2}
Let 
\begin{center}
\begin{tikzpicture}[descr/.style={fill=white,inner sep=2.5pt}]
\matrix(a)[matrix of math nodes, row sep=3em, column sep=2.5em, text height=1.5ex, text depth=0.25ex
]
{
   X' & Y' & P' \\
   X  & Y  & P \\
};
\path[->,font=\scriptsize] (a-1-1)
edge node[above]{}  (a-1-2);
\path[->,font=\scriptsize] (a-1-2)
edge node[above]{}  (a-1-3);
\path[->,font=\scriptsize] (a-2-1)
edge node[above]{}  (a-2-2);
\path[->,font=\scriptsize] (a-2-2)
edge node[above]{}  (a-2-3);
\path[->,font=\scriptsize] (a-1-2)
edge node[right]{}  (a-2-2);
\path[->,font=\scriptsize] (a-1-3)
edge node[right]{$u$}  (a-2-3);
\path[->,font=\scriptsize] (a-1-1)
edge node[right]{$f$}  (a-2-1);
\end{tikzpicture}
\end{center}
be an \'{e}tale morphism of smooth $S$-frames. Also assume that $f$ is \'{e}tale. Let $E$ be a coherent $j_X^{\dagger} \Op_{]Y[_P}$-module with an integrable connection over $K$. Choose two closed points $x' \in X'$ and $x \in X$ such that $f^{-1}(x) = \{x'\}$ and such that $f$ defines an isomorphism $k(x) \xlongrightarrow{\sim} k(x')$ on the residue fields. Then the base change map
\[
u \star \colon \underline{\Gamma}^{\dagger}_{\{x\}} E \otimes_{\Op_{]Y[_P}} \Omega^{\bullet}_{]Y[_P / K} \longrightarrow \left( \R u_{K *} \right) \underline{\Gamma}^{\dagger}_{\{x'\}} u^{\dagger} E \otimes_{\Op_{]Y'[_{P'}}} \Omega^{\bullet}_{]Y'[_{P'} / K}
\]
is an isomorphism.
\end{proposition}

\begin{proof}
Let $Q' \subset P'$ be an open neighbourhood of $X'$ such that the restriction $u_{\mid Q'}$ is \'{e}tale. Then define $X'' = X \times_P Q'$ and $Y'' = Y \times_P P'$. We claim that the canonical map $X' \rightarrow Y''$ is an open immersion. Since this morphism factors through $X''$ and since it is easy to see that the canonical map $X'' \rightarrow Y''$ is an open immersion, it suffices to show that the canonical morphism $\alpha \colon X' \rightarrow X''$ is an open immersion. It is clear that $\alpha$ is an immersion. Now consider the projection morphism $\beta \colon X'' \rightarrow X$. Since $f = \beta \circ \alpha$ and since $\beta$ is \'{e}tale by construction, it follows that $\alpha$ is \'{e}tale as well. This shows that $\alpha$ is indeed an open immersion, proving our claim. The fact that $X' \rightarrow Y''$ is an open immersion allows us to factor our morphism of frames as follows:
\begin{center}
\begin{tikzpicture}[descr/.style={fill=white,inner sep=2.5pt}]
\matrix(a)[matrix of math nodes, row sep=3em, column sep=2.5em, text height=1.5ex, text depth=0.25ex
]
{
   & Y' &  \\
 X' & Y'' & P' \\
 X  & Y  & P  \\
};
\path[->,font=\scriptsize] (a-2-1)
edge node[above]{}  (a-2-2);
\path[->,font=\scriptsize] (a-2-1)
edge node[right]{}  (a-3-1);
\path[->,font=\scriptsize] (a-2-2)
edge node[above]{}  (a-2-3);
\path[->,font=\scriptsize] (a-3-2)
edge node[above]{}  (a-3-3);
\path[->,font=\scriptsize] (a-1-2)
edge node[right]{}  (a-2-2);
\path[->,font=\scriptsize] (a-2-2)
edge node[right]{}  (a-3-2);
\path[->,font=\scriptsize] (a-2-3)
edge node[right]{$u$}  (a-3-3);
\path[->,font=\scriptsize] (a-3-1)
edge node[right]{}  (a-3-2);
\path[->,font=\scriptsize] (a-1-2)
edge node[above]{}  (a-2-3);
\path[->,font=\scriptsize] (a-2-1)
edge node[above]{}  (a-1-2);
\end{tikzpicture}
\end{center}
Recall that we assumed the formal schemes $P'$ and $P$ to be topologically of finite type over $\mathcal{V}$. It follows that these formal schemes are Noetherian topological spaces, since their closed fibers are of finite type over $k$. In particular we see that $u$ is quasi-compact. Therefore the morphism $u_K \colon P'_K \rightarrow P_K$ on the generic fibers is quasi-compact as well. Since the rightmost square in the diagram above is Cartesian we find that $u_K^{-1} \left( ]Y[_{P} \right) = ]Y''[_{P'}$ according to \cite[Proposition 2.2.6]{lestum}. It follows that the induced morphism on tubes $u_K \colon ]Y''[_{P'} \rightarrow ]Y[_P$ is quasi-compact. This allows us to apply proposition \ref{theorem_base_change_etale} to the lower part of the diagram. Now observe that the morphism $Y' \rightarrow Y''$ is a closed immersion, hence proper. So according to \cite[Proposition 6.5.3]{lestum}, the base change map that is associated to the upper part of the diagram is an isomorphism as well.
\end{proof}

The key idea for the proof of proposition \ref{theorem_base_change_main} is to show that an \'{e}tale map $f \colon X' \rightarrow X$ has an \'{e}tale realization, at least after shrinking $X'$ and $X$. In this way one reduces the problem to proposition \ref{theorem_base_change_etale_2}. The proof of this fact relies on a number of geometric results that we discuss below.

\begin{proposition}
\label{reductions}
~\begin{enumerate}[label=\roman*)]
\item
Consider a proper morphism of frames
\begin{center}
\begin{tikzpicture}[descr/.style={fill=white,inner sep=2.5pt}]
\matrix(a)[matrix of math nodes, row sep=3em, column sep=2.5em, text height=1.5ex, text depth=0.25ex
]
{
 X' & Y' & P' \\
 X  & Y & P \\
};
\path[->,font=\scriptsize] (a-1-1)
edge node[above]{}  (a-1-2);
\path[->,font=\scriptsize] (a-1-2)
edge node[above]{}  (a-1-3);
\path[->,font=\scriptsize] (a-1-1)
edge node[right]{$f$}  (a-2-1);
\path[->,font=\scriptsize] (a-2-2)
edge node[above]{}  (a-2-3);
\path[->,font=\scriptsize] (a-1-2)
edge node[right]{}  (a-2-2);
\path[->,font=\scriptsize] (a-1-3)
edge node[right]{}  (a-2-3);
\path[->,font=\scriptsize] (a-2-1)
edge node[right]{}  (a-2-2);
\end{tikzpicture}
\end{center}
where $f$ is quasi-projective. Then we can blow up a closed subvariety of $Y'$ outside $X'$ in $P'$ to obtain a diagram
\begin{center}
\begin{tikzpicture}[descr/.style={fill=white,inner sep=2.5pt}]
\matrix(a)[matrix of math nodes, row sep=3em, column sep=2.5em, text height=1.5ex, text depth=0.25ex
]
{
   & \widetilde{Y'} & \widetilde{P'} \\
 X' & Y' & P' \\
 X  & Y  & P  \\
};
\path[->,font=\scriptsize] (a-2-1)
edge node[above]{}  (a-1-2);
\path[->,font=\scriptsize] (a-2-1)
edge node[above]{}  (a-2-2);
\path[->,font=\scriptsize] (a-2-1)
edge node[right]{$f$}  (a-3-1);
\path[->,font=\scriptsize] (a-1-2)
edge node[above]{}  (a-1-3);
\path[->,font=\scriptsize] (a-2-2)
edge node[above]{}  (a-2-3);
\path[->,font=\scriptsize] (a-3-2)
edge node[above]{}  (a-3-3);
\path[->,font=\scriptsize] (a-1-2)
edge node[right]{}  (a-2-2);
\path[->,font=\scriptsize] (a-1-3)
edge node[right]{}  (a-2-3);
\path[->,font=\scriptsize] (a-2-2)
edge node[right]{}  (a-3-2);
\path[->,font=\scriptsize] (a-2-3)
edge node[right]{}  (a-3-3);
\path[->,font=\scriptsize] (a-3-1)
edge node[right]{}  (a-3-2);
\end{tikzpicture}
\end{center}
where the composition $\widetilde{Y'} \rightarrow Y$ is projective.
\item
Consider a strict morphism of frames 
\begin{center}
\begin{tikzpicture}[descr/.style={fill=white,inner sep=2.5pt}]
\matrix(a)[matrix of math nodes, row sep=3em, column sep=2.5em, text height=1.5ex, text depth=0.25ex
]
{
 X' & Y' & P' \\
 X & Y & P \\
};
\path[->,font=\scriptsize] (a-1-1)
edge node[above]{}  (a-1-2);
\path[->,font=\scriptsize] (a-1-2)
edge node[above]{}  (a-1-3);
\path[->,font=\scriptsize] (a-2-1)
edge node[above]{}  (a-2-2);
\path[->,font=\scriptsize] (a-2-2)
edge node[above]{}  (a-2-3);
\path[->,font=\scriptsize] (a-1-1)
edge node[right]{}  (a-2-1);
\path[->,font=\scriptsize] (a-1-2)
edge node[right]{}  (a-2-2);
\path[->,font=\scriptsize] (a-1-3)
edge node[right]{$u$}  (a-2-3);
\end{tikzpicture}
\end{center}
where $u$ is a formal blowing up. Then the map $u_K \colon ]Y'[_{P'} \rightarrow ]Y[_P$ is an isomorphism. Moreover, any admissible open neighbourhood $V$ of $]Y[_P$ is a strict neighbourhood of $]X[_P$ in $]Y[_P$ if and only if $u_K^{-1}(V)$ is a strict neighbourhood of $]X'[_{P'}$ in $]Y'[_{P'}$. That is, giving a $j_{X}^{\dagger} \Op_{]Y[_P}$-module amounts to the same thing as giving a $j_{X'}^{\dagger} \Op_{]Y'[_{P'}}$-module.
\item
Consider a frame $(X \subset Y \subset P)$ together with a diagram 
\begin{center}
\begin{tikzpicture}[descr/.style={fill=white,inner sep=2.5pt}]
\matrix(a)[matrix of math nodes, row sep=3em, column sep=2.5em, text height=1.5ex, text depth=0.25ex
]
{
 X' & Y' \\
 X  & Y  & P\\
};
\path[->,font=\scriptsize] (a-1-1)
edge node[above]{}  (a-1-2);
\path[->,font=\scriptsize] (a-2-2)
edge node[above]{}  (a-2-3);
\path[->,font=\scriptsize] (a-1-1)
edge node[right]{$f$}  (a-2-1);
\path[->,font=\scriptsize] (a-2-1)
edge node[above]{}  (a-2-2);
\path[->,font=\scriptsize] (a-1-2)
edge node[right]{$g$}  (a-2-2);
\end{tikzpicture}
\end{center}
where the map $X' \rightarrow Y'$ is an open immersion, $f$ is an \'{e}tale morphism and $g$ is projective. Then locally on $(X \subset Y \subset P)$, there exists a closed subscheme $Y'' \subset Y'$ containing $X'$ such that the map $g_{\mid Y''}$ extends to a proper \'{e}tale morphism of frames
\begin{center}
\begin{tikzpicture}[descr/.style={fill=white,inner sep=2.5pt}]
\matrix(a)[matrix of math nodes, row sep=3em, column sep=2.5em, text height=1.5ex, text depth=0.25ex
]
{
 X' & Y'' & P' \\
 X & Y & P \\
};
\path[->,font=\scriptsize] (a-1-1)
edge node[above]{}  (a-1-2);
\path[->,font=\scriptsize] (a-1-2)
edge node[above]{}  (a-1-3);
\path[->,font=\scriptsize] (a-2-1)
edge node[above]{}  (a-2-2);
\path[->,font=\scriptsize] (a-2-2)
edge node[above]{}  (a-2-3);
\path[->,font=\scriptsize] (a-1-1)
edge node[right]{$f$}  (a-2-1);
\path[->,font=\scriptsize] (a-1-2)
edge node[right]{$g$}  (a-2-2);
\path[->,font=\scriptsize] (a-1-3)
edge node[right]{$u$}  (a-2-3);
\end{tikzpicture}
\end{center}
\end{enumerate}
\end{proposition}

\begin{proof}
~\begin{enumerate}[label=\roman*)]
\item
Apply \cite[Corollaire 5.7.14]{chow} to the morphism $Y' \rightarrow Y$ and the open subset $X' \subset Y'$, which is quasi-projective over $Y$.
\item
This follows from \cite[Corollary 2.2.7]{lestum} and \cite[Proposition 3.1.13]{lestum}.
\item
The composition $Y' \rightarrow Y \rightarrow P$ can be factored through a closed immersion $Y' \rightarrow \Proj^N_P$ for some $N$. It now suffices to show that the morphism $i$ in the diagram below is a regular immersion. The rest of the proof is analogous to \cite[Lemma 6.5.1]{lestum}.
\begin{center}
\begin{tikzpicture}[descr/.style={fill=white,inner sep=2.5pt}]
\matrix(a)[matrix of math nodes, row sep=3em, column sep=2.5em, text height=1.5ex, text depth=0.25ex
]
{
 X' & & \\
 & \Proj^N_P \times_P X & \Proj^N_P \\
 & X & P \\
};
\path[->,font=\scriptsize] (a-2-2)
edge node[above]{}  (a-2-3);
\path[->,font=\scriptsize] (a-3-2)
edge node[below]{}  (a-3-3);
\path[->,font=\scriptsize] (a-2-2)
edge node[right]{}  (a-3-2);
\path[->,font=\scriptsize] (a-2-3)
edge node[right]{}  (a-3-3);
\path[->,font=\scriptsize] (a-1-1)
edge node[above]{$i$}  (a-2-2);
\path[->,font=\scriptsize] (a-1-1)
edge[bend left=15] node[above]{}  (a-2-3);
\path[->,font=\scriptsize] (a-1-1)
edge[bend right=15] node[left]{$f$}  (a-3-2);
\end{tikzpicture}
\end{center}
First note that $i$ is an immersion, since the morphisms $X' \rightarrow \Proj^N_P$ and $\Proj^N_P \times_P X \rightarrow \Proj^N_P$ are immersions. Also, $\Proj^N_P \times_P X \rightarrow X$ is smooth since it is obtained by base extension from a smooth morphism. Since $f$ is a local complete intersection morphism it follows from \cite[Corollary 6.3.22]{liu} that $i$ is indeed regular.
\end{enumerate}
\end{proof}

With all the preliminary work, the proof of proposition \ref{theorem_base_change_main} becomes very similar to the proof of \cite[Proposition 6.5.3]{lestum}. 

\begin{proof}[Proof of Proposition \ref{theorem_base_change_main}]
First note that we can always replace $X'$ and $X$ by open neighbourhoods $U_{x'}$, $U_x$ of $x'$ resp.\ of $x$ such that $f(U_{x'}) \subset U_x$. Indeed, the base change map coming from the diagram
\begin{center}
\begin{tikzpicture}[descr/.style={fill=white,inner sep=2.5pt}]
\matrix(a)[matrix of math nodes, row sep=3em, column sep=2.5em, text height=1.5ex, text depth=0.25ex
]
{
 U_{x} &  &  \\
 X & Y & P \\
};
\path[->,font=\scriptsize] (a-1-1)
edge node[above]{}  (a-2-2);
\path[->,font=\scriptsize] (a-2-1)
edge node[above]{}  (a-2-2);
\path[->,font=\scriptsize] (a-2-2)
edge node[above]{}  (a-2-3);
\path[->,font=\scriptsize] (a-1-1)
edge node[right]{}  (a-2-1);
\end{tikzpicture}
\end{center}
is simply the canonical map $\underline{\Gamma}_{\{x\}}^{\dagger} \rightarrow \underline{\Gamma}_{\{x\}}^{\dagger} j_{U_x}^{\dagger}$ applied to the de Rham complex of $E$. This is an isomorphism by \cite[Proposition 5.2.12]{lestum}. A similar argument holds for the inclusion $U_{x'} \hookrightarrow X'$ and the $j_{X'}^{\dagger} \Op_{]Y'[_{P'}}$-module $u^{\dagger} E$. We may also replace $Y'$ by a closed subscheme that contains $X'$. By \cite[Proposition 6.5.3]{lestum} this does not alter the base change map either. We will refer to a combination of these two operations as a \emph{shrinking} of the data. The fact that a shrinking of the data does not alter the base change map can be used to reduce the problem to the case where $f$ has an \'{e}tale realization. Indeed, after replacing $X'$ and $X$ by open neighbourhoods of $x'$ resp.\ of $x$ we may assume that $f$ is an affine morphism, hence quasi-projective. By the first two points of proposition \ref{reductions} we then reduce to the case where $g$ is projective. After some more shrinking of $X'$ and $Y'$ we may use the third point of proposition \ref{reductions} to obtain an \'{e}tale morphism of frames
\begin{center}
\begin{tikzpicture}[descr/.style={fill=white,inner sep=2.5pt}]
\matrix(a)[matrix of math nodes, row sep=3em, column sep=2.5em, text height=1.5ex, text depth=0.25ex
]
{
   X' & Y' & P'' \\
   X  & Y  & P \\
};
\path[->,font=\scriptsize] (a-1-1)
edge node[above]{}  (a-1-2);
\path[->,font=\scriptsize] (a-1-2)
edge node[above]{}  (a-1-3);
\path[->,font=\scriptsize] (a-2-1)
edge node[above]{}  (a-2-2);
\path[->,font=\scriptsize] (a-2-2)
edge node[above]{}  (a-2-3);
\path[->,font=\scriptsize] (a-1-2)
edge node[right]{$g$}  (a-2-2);
\path[->,font=\scriptsize] (a-1-3)
edge node[right]{$v$}  (a-2-3);
\path[->,font=\scriptsize] (a-1-1)
edge node[right]{$f$}  (a-2-1);
\end{tikzpicture}
\end{center} 
Now consider the diagonal embedding $Y' \hookrightarrow P''' = P' \times_P P''$ and let $p_1 \colon P''' \rightarrow P'$ and $p_2 \colon P''' \rightarrow P''$ denote the projection maps. By construction we have that
\begin{equation}
\label{eq_commute}
u \circ p_1 = v \circ p_2.
\end{equation}
Also, $p_1$ and $p_2$ are smooth since they are obtained by base extension from $v$ resp.\ from $u$. By the identity (\ref{eq_commute}) it is now sufficient to prove that the base change maps that are associated to $v$ and to the diagrams
\begin{center}
\begin{tikzpicture}[descr/.style={fill=white,inner sep=2.5pt}]
\matrix(a)[matrix of math nodes, row sep=3em, column sep=2.5em, text height=1.5ex, text depth=0.25ex
]
{
      &    & P'''	     \\
   X' & Y' &             \\
      &    & P', P'' \\
};
\path[->,font=\scriptsize] (a-2-1)
edge node[above]{}  (a-2-2);
\path[->,font=\scriptsize] (a-2-2)
edge node[above]{}  (a-1-3);
\path[->,font=\scriptsize] (a-2-2)
edge node[above]{}  (a-3-3);
\path[->,font=\scriptsize] (a-1-3)
edge node[right]{$p_1, p_2$}  (a-3-3);
\end{tikzpicture}
\end{center} 
are isomorphisms. In proposition \ref{theorem_base_change_etale_2} we have already proved that the base change map for the \'{e}tale morphism of frames $v$ is an isomorphism. For the two morphisms of frames $p_1$ and $p_2$ it follows directly from \cite[Proposition 6.5.3]{lestum}. 
\end{proof}

\singlespacing
\bibliographystyle{bib_custom}
\bibliography{refs}

\end{document}